\documentclass[a4paper,reqno, 10pt]{amsart}   	
\usepackage[DIV=12, oneside]{typearea}			
\usepackage[utf8]{inputenc}
\usepackage[T1]{fontenc}
\usepackage[english]{babel}
\usepackage[centertags]{amsmath}
\usepackage{amstext,amssymb,amsopn,amsthm}
\usepackage{mathrsfs}
\usepackage{dsfont}
\usepackage{bbm}
\usepackage{thmtools}
\usepackage{graphicx}
\usepackage[titletoc,title]{appendix}
\usepackage{etexcmds}

\usepackage[backgroundcolor=white, bordercolor=blue,
linecolor=blue]{todonotes}

\usepackage{tikz}

\usepackage{pgfplots}
\pgfplotsset{compat=newest}
\usetikzlibrary{calc, patterns,angles,quotes}
\usepgfplotslibrary{fillbetween}
\definecolor{c595959}{RGB}{89,89,89}
\definecolor{c5a5a5a}{RGB}{90,90,90}

\usepackage[colorlinks=true, linkcolor=black, citecolor=black]{hyperref}
\usepackage{enumitem}
\setlist[enumerate]{itemsep=0mm}

\addto\extrasenglish{}
\addto\extrasenglish{}
\addto\extrasenglish{}

\theoremstyle{plain}
\declaretheorem[title=Theorem, parent=section]{theorem}
\declaretheorem[title=Lemma,sibling=theorem]{lemma}
\declaretheorem[title=Proposition,sibling=theorem]{proposition}
\declaretheorem[title=Corollary,sibling=theorem]{corollary}

\theoremstyle{definition}
\declaretheorem[title=Definition,sibling=theorem]{definition}
\declaretheorem[title=Remark,sibling=theorem]{remark}
\declaretheorem[title=Remark, numbered=no]{remark*}

\declaretheorem[title=Assumption, numbered=no]{assumption*}

\numberwithin{equation}{section}

\newcommand{\N}{\mathds{N}}
\newcommand{\R}{\mathds{R}}

\def\hmath$#1${\texorpdfstring{{\rmfamily\textit{#1}}}{#1}}

\newcommand{\cB}{\mathcal{B}}
\newcommand{\cW}{\mathcal{W}}

\newcommand{\eps}{\varepsilon}

\newcommand{\loc}{\mathrm{loc}}

\newcommand{\BIGOP}[1]
{
	\mathop{\mathchoice%
		{\raise-0.22em\hbox{\huge $#1$}}%
		{\raise-0.05em\hbox{\Large $#1$}}{\hbox{\large $#1$}}{#1}}}

\def\XXint#1#2#3{{\setbox0=\hbox{$#1{#2#3}{\int}$}
		\vcenter{\hbox{$#2#3$}}\kern-.5\wd0}}

\newcommand{\BIGboxplus}{\mathop{\mathchoice%
		{\raise-0.35em\hbox{\huge $\boxplus$}}%
		{\raise-0.15em\hbox{\Large $\boxplus$}}{\hbox{\large $\boxplus$}}{\boxplus}}}

\DeclareMathOperator{\lin}{lin}
\DeclareMathOperator{\dist}{dist}

\DeclareMathOperator{\Id}{Id}
\DeclareMathOperator{\diam}{diam}

\DeclareMathOperator{\supp}{supp}

\renewcommand{\d}{\textnormal{d}}

\newcommand{\1}{{\mathbbm{1}}}

\usepackage{esint}
\newcommand{\norm}[1]{\left\lVert#1\right\rVert} 
\newcommand{\abs}[1]{\ensuremath{\left\vert#1\right\vert}} 
 
\newcommand{\ie}{i.e.\,}

\newcommand{\eg}{e.g.\,}

\newcommand{\distw}[2]{\dist(#1, #2)} 

\makeatletter
\providecommand\@dotsep{5}
\def\listtodoname{List of Todos}
\def\listoftodos{\@starttoc{tdo}\listtodoname}
\makeatother

\begin{document}
	\allowdisplaybreaks
	\title{The Dirichlet Problem for L\'{e}vy-stable operators with $L^2$-data}
	
		\author{Florian Grube, Thorben Hensiek, Waldemar Schefer}

	\address{Fakultät für Mathematik, Universität Bielefeld, Postfach 10 01 31, 33501 Bielefeld, Germany}
	\email{fgrube@math.uni-bielefeld.de}
	\email{thensiek@math.uni-bielefeld.de}
	\email{wschefer@math.uni-bielefeld.de}
	
	\makeatletter
	\@namedef{subjclassname@2020}{%
		\textup{2020} Mathematics Subject Classification}
	\makeatother
	
	

	\begin{abstract} 
		We prove Sobolev regularity for distributional solutions to the Dirichlet problem for generators of $2s$-stable processes and exterior data, inhomogeneity in weighted $L^2$-spaces. This class of operators includes the fractional Laplacian. For these rough exterior data the theory of weak variational solutions is not applicable. Our regularity estimate is robust in the limit $s\to 1-$ which allows us to recover the local theory. 
	\end{abstract}

	\maketitle
	
	\section{Introduction}\label{sec:intro}
	The study of the existence, uniqueness, and regularity of harmonic functions in domains with prescribed boundary values has a long history. In the late 1970s and 1980s significant progress was made on the Dirichlet problem with irregular boundary data. It is well known that the problem \begin{align}\label{eq:dirichlet_laplace}
	\begin{split}
	-\Delta u &= 0 \text{ in } \Omega,\\
	u &=g \text{ on } \partial\Omega,
	\end{split}
	\end{align}
	admits a unique solution $u$ in $H^{1/2}(\Omega)$ for $g$ in $L^2(\partial\Omega)$ and a sufficiently regular domain $\Omega\subset \R^d$, see e.g.\ \cite{KenigDirichlet}. One goal in this field of study is to lower the assumptions on the regularity of the boundary $\partial \Omega$, cf. our overview in \autoref{sec:literature}. A major difficulty in studying \eqref{eq:dirichlet_laplace} in rough domains is the lack of Poisson kernel estimates near the boundary of $\Omega$. In addition, the low regularity of the boundary datum $g$ prohibits the use of a variational ansatz in $H^1(\Omega)$ since the classical trace and extension theorems restrict us naturally to boundary data in $H^{1/2}(\partial \Omega)$. \smallskip 
	
	In the last decades, nonlocal operators enjoyed notable attention. Naturally motivated by the Courrège theorem, see \cite{Co65}, which characterizes linear operators that satisfy the positive maximum principle, these operators appear in probability theory, fluid mechanics and mathematical physics. In this work, we study the subclass of $2s$-stable, nondegenerate integro-differential operators
	\begin{align}
		A_s u(x):= \text{p.v.} \int\limits_{\mathbb{R}^{d}} (u(x)-u(x+h)) \,\nu_s(\d h) \quad \text{for $u:\R^d \to \R$,} \label{eq:definition_A_s}
	\end{align}
	with $s\in (0,1)$, a measure $\mu$ on the unit sphere $\mathbb S^{d-1}$, and the L{\'e}vy measure given in polar coordinates via
	\begin{align}
		\nu_s(U):= (1-s) \int\limits_{\mathbb{R}}\int\limits_{\mathbb S^{d-1}} \1_U(r\theta) \abs{r}^{-1-2s}  \mu(\d\theta) \d r \quad \text{for $U\in \mathcal{B}(\mathbb{R}^d)$.}\label{eq:levy_measure}
	\end{align}	
	We assume that the measure $\mu$ is finite, i.e.\ there exists $\Lambda>0$ such that $\mu(S^{d-1})\le \Lambda$, and the following nondegeneracy condition: There exist $0<\lambda\le \Lambda$ such that
	\begin{align}\label{eq:ellip_conditions}
		\lambda\le \inf\limits_{\omega\in \mathbb S^{d-1}} \int\limits_{\mathbb S^{d-1}}\abs{\omega\cdot \theta}^{2s} \mu(\d\theta).
	\end{align}
	The measure $\mu$ weights the directions in which the operator $A_s$ acts. In particular, the assumption \eqref{eq:ellip_conditions} ensures that $\mu$ is not supported on a hyperplane.
	
	\smallskip
	
	The class of $2s$-stable operators includes the fractional Laplacian $(-\Delta)^s$ which corresponds to $\mu$ being a multiple of the surface measure on the sphere. The assumption \eqref{eq:ellip_conditions} yields the comparability of the Fourier symbol of $A_s$ to that of the fractional Laplacian, \ie $\abs{\xi}^{2s}$. 
	
	\smallskip 
	
	These operators are translation invariant, symmetric and exhibit the scaling behavior $A_s [u(\lambda \cdot)](x)= \lambda^{2s} A_s [ u](\lambda x)$, $x\in \R^d, \lambda>0$. Furthermore, elliptic second-order operators $a^{ij}\partial_{ij}$ with constant coefficients $a^{ij}$ appear in the limit of $A_s$ as $s\to 1-$, cf.\ \cite{FKV20}. 
	
	\smallskip

	Another motivation to consider operators of the form \eqref{eq:definition_A_s} arises from the field of stochastic processes. More precisely, the operators $A_s$ under the assumption \eqref{eq:ellip_conditions} are the generators of symmetric $2s$-stable nondegenerate processes, see \eg \cite{Sat99}.
	
	In analogy to the classical problem \eqref{eq:dirichlet_laplace}, we study the exterior data problems
	\begin{align}\label{eq:main_equation}
		\begin{split}
			A_s u &= f \text{ in } \Omega,\\
			u &= g \text{ on }\Omega^c,
		\end{split}
	\end{align}
	for given $f:\Omega \to \R$ and $g:\Omega^c\to \R$ and a bounded $C^{1,\alpha}$-domain $\Omega \subset \R^d$. More precisely, we study the existence, uniqueness and Sobolev-regularity of distributional solutions to the equation \eqref{eq:main_equation}. We say that a function $u:\R^d\to \R$ is a distributional solution to \eqref{eq:main_equation} if it satisfies
	\begin{equation}\label{eq:distributional_A_s}
		\int\limits_{\R^d} u(x)\,A_s (\eta)(x)\d x = \int\limits_{\Omega} f(x)\, \eta(x)\d x \quad \text{for all }\eta \in C_c^\infty(\Omega)
	\end{equation}
	and $u=g$ on $\Omega^c$, as well as some appropriate integrability properties.	
	
	\smallskip 
	
	Similar difficulties as in the classical setup \eqref{eq:dirichlet_laplace}, described above, occur. In particular, boundary estimates of the corresponding Green and Poisson kernels for the linear operators $A_s$ in general $C^{1,\alpha}$-domains are not known. Furthermore, exterior data in some weighted $L^2$-space are not sufficient to allow for a variational approach, see \eg \cite{GrHe22b}.
	
	\smallskip
	
	The goal of this work is twofold. Firstly, we want to provide a proof of the existence, uniqueness and Sobolev-regularity of distributional solutions to \eqref{eq:main_equation} for $f,g$ in some weighted $L^2$-spaces that is both short and easy to generalize. In particular, we will avoid any potential theory. Secondly, we want to provide a regularity theory that is continuous in the order of differentiation $s$, \ie the classical result for second-order operators, as described above, is retrieved in the limit $s\to 1-$. This entails a careful study of all estimates in terms of $s$. Note that the data $g$ in the classical problem \eqref{eq:dirichlet_laplace} is given on the topological boundary of $\Omega$ which is in align with the locality of the Laplacian. In turn, the nonlocal problem \eqref{eq:main_equation} requires exterior data on the complement of $\Omega$. Thus, an appropriate choice of exterior data $g$ is essential to achieve a ‘continuous’ theory. For a big subclass of the operators under consideration, like the fractional Laplacian $(-\Delta)^s$, a fitting weighted $L^2$-space is already introduced in the previous works \cite{GrHe22b} and \cite{GrKa23} which deal with robust nonlocal trace and extension results.
	
	\smallskip 
	
	\subsection{The main result} Throughout this work, we assume $\Omega\subset \R^d$ to be a bounded $C^{1,\alpha}$-domain for some $\alpha \in (0,1)$. In order to study the well-posedness of the distributional problem \eqref{eq:main_equation}, it is essential that the integrals in \eqref{eq:distributional_A_s} exist. In particular, a distributional solution $u$ needs to be integrable against $A_s (\eta)$ for any $\eta \in C_c^\infty(\Omega)$. A corresponding tail weight was introduced in \cite{GrHe22a}: For any $s\in (0,1)$ we define
	\begin{equation}\label{eq:nustar}
	\nu_s^\star(x):= (1-s)\int\limits_{\mathbb{R}}\int\limits_{\mathbb S^{d-1}}  \1_{\Omega} (x+r\theta) (1+\abs{r})^{-1-2s}   \mu(\d\theta)\d r.
	\end{equation}
	This weight decays like $A_s(\eta)$ at infinity and is locally bounded. Thus, $u\in L^1(\R^d; \nu_s^\star)$ is a sufficient and necessary condition to study distributional solutions of \eqref{eq:main_equation}.  Instead of $\nu_s^\star$, many authors use $\widetilde{\nu}^\star_s(x):=(1-s)/(1+\abs{x})^{d+2s}$. It was proved in \cite[Lemma A.1]{GrHe22a} that in some cases, including the fractional Laplacian, this weight is comparable to $\nu_s^\star$. We want to emphasize that the choice $\widetilde{\nu}_s^\star$ is not optimal whenever the measure $\mu$ is not supported on the whole sphere. For example the sum of Dirac measures on an orthogonal basis of $\R^d$, $\mu= \sum_{i=1}^{d} \delta_{e_{i}}$, corresponds to the operator which is a sum of one-dimensional fractional Laplacians $\sum_{i=1}^d (-\partial_i^2)^s$ in all coordinate directions and the measure $\nu_s^\star(x)\d x$ is supported only in a neighborhood of the coordinate axes depending on $\Omega$, i.e.\ $\supp \nu_s^\star= \{ x+re_i\,|\, x\in \Omega, r\in \R, i\in \{1,\dots, d\} \}$. 
	
	\smallskip
	
	Next, we introduce the weighted $L^2$-spaces for the exterior data $g$ in the problem \eqref{eq:main_equation}. From the above discussion, it is clear that the weight needs to behave like $\nu_s^\star$ at infinity. Additionally, one goal of this work is to provide a theory that is continuous in terms of $s$. Thereby, we need to retrieve the space $L^2(\partial \Omega)$ in the limit $s\to 1-$. In the case of the fractional Laplacian, this was achieved in \cite{GrHe22b} and \cite{GrKa23} where trace and extension results for a variational approach was proved. In the following definition we introduce a weight in the spirit of \cite{GrHe22b} that is adapted to the operators $A_s$.
	
	\begin{definition}\label{def:mu_s}
		Let $s\in (0,1)$. The measure $\tau_s$ on $\cB(\R^d)$ is defined by $\tau_s(\d x):= \tau_s(x)\d x$ where
		\begin{equation}\label{eq:def_tau_s}
		\tau_s(x):= \frac{1-s}{d_x^s}\1_{\Omega^1\cap \supp(\nu_s^\star)}(x) +\nu_s^\star(x),
		\end{equation}
		 $\Omega^1:= \{ x\in \Omega^c\,|\, d_x<1 \}$, and $d_x := \distw{x}{\partial\Omega}$.
	\end{definition}

	Now, we state our main result. 
	
	\begin{theorem}\label{th:main_result}
		Let $\Omega\subset \R^d$ be a bounded $C^{1,\alpha}$-domain, $\alpha \in (0,1)$, $s_\star\in (0,1)$. For any $s\in (s_\star,1)$, $f\in L^2(\Omega; d_x^{2s})$, $g \in L^2(\Omega^c;\tau_s)$, there exists a unique distributional solution $u\in H^{s/2}(\Omega) \cap L^2(\R^d; \nu_s^\star)$ to \eqref{eq:main_equation}. Furthermore, the estimate 
		\begin{align}\label{eq:main_estimate_theorem}
		\begin{split}
		\norm{u}_{H^{s/2}(\Omega)}^2&\le \norm{u}_{L^2(\Omega)}^2 +C (1-s) \iint\limits_{\Omega\, \Omega} \big(d_x\wedge d_y\big)^s \frac{(u(x)-u(y))^2}{\abs{x-y}^{d+2s}} \d y \d x\\
		&\qquad \qquad\le C^2 \Big(\norm{f}_{L^2(\Omega;d_x^{2s})}^2+\norm{g}_{L^2(\Omega^c;\tau_s)}^2\Big)
		\end{split}
		\end{align} 
		holds for some constant $C=C(d,\Omega, \alpha, s_\star, \lambda, \Lambda)\ge 1$.
	\end{theorem}

	This result is new, even in the case of the fractional Laplacian on $C^{1,\alpha}$-domains. Only recently in \cite{KiRy23}, \autoref{th:main_result} was proved for $(-\Delta)^s$ on $C^{1,1}$-domains but the dependency of the constant in \eqref{eq:main_estimate_theorem} on $s$ is unclear. 
	\begin{remark}
		The first inequality in \eqref{eq:main_estimate_theorem} holds true for all measurable $u :\Omega \to \mathbb{R}$ for which the right-hand side is finite, see \autoref{lem:comparison_forms}. It is a fractional analogue of \cite[Theorem 4.1]{KenigDirichlet}. 
	\end{remark}
		
	\begin{remark}
		The proof of \autoref{th:main_result} reveals that
		\begin{equation*}
		(1-s)\int_\Omega  d_x^{s} \int_{\mathbb{R}} \int_{\mathbb S^{d-1}} \1_{\Omega^c}(x+r\theta)\frac{(u(x)-g(x+r\theta))^2}{\abs{r}^{1+2s}} \d \mu(\theta) \d r \d x
		\end{equation*}
		is bounded by the right-hand side of \eqref{eq:main_estimate_theorem} for the solution $u\in H^{s/2}(\Omega)$. This term describes how the solution $u$ in $\Omega$ meets the exterior data $g$ at the boundary $\partial \Omega$. This is a purely nonlocal phenomenon since this term vanishes in the limit $s\to 1-$. 
	\end{remark}
	
	\begin{remark}[Optimality] Let $\Omega$ be any bounded $C^{1,1}$-domain, $g(x):= \1_{\Omega^1}(x) d_x^{-1+s+\varepsilon}$, and $u(x):= \int_{\Omega^c} g(y)P_s(x,y) \d y$, where $P_s$ is the Poisson kernel related to the fractional Laplacian. Then $u$ solves $(-\Delta)^s u = 0$ in $\Omega$ and $u = g$ on $\Omega^c$. The exterior data $g$ is in $L^2(\Omega^c;\tau_s)$ if and only if $\varepsilon > (1-s)/2$. A short calculation, using Poisson kernel estimates, yields\begin{equation*}
		\int_{\Omega} \frac{u(x)^2}{d_x^s} \d x \ge c \int_{\Omega} d_x^{-2+s+2\varepsilon} \d x.
		\end{equation*}
		The right-hand side is finite if and only if $\varepsilon > (1-s)/2$. \smallskip
		
		If we pick $\varepsilon \le  (1-s)/2$, then $g$ is not in $L^2(\Omega^c;\tau_s)$ and Hardy's inequality implies that the solution $u$ does not belong to $H^{s/2}(\Omega)$, see \cite[Theorem 1.1 (17)]{Dyd04} or \cite[Theorem 2.3]{ChSo03}. \smallskip
		
		If we choose $\varepsilon = (1-s)/2 + \delta/2$ for $\delta >0$, then $g \in L^2(\Omega^c;\tau_s)$, but 
		\begin{equation*}
		\int_{\Omega} \frac{u(x)^2}{d_x^{s+2\delta}} \d x = \infty.
		\end{equation*}
		Hence, $u \in H^{s/2}(\Omega)$ by \autoref{th:main_result} but $u \notin H^{s/2 +\delta}(\Omega)$. 
	\end{remark}
	
	\begin{remark}
		Under stronger assumptions on the exterior data $g$, the variational approach implies the $H^s(\Omega)$-regularity. This approach requires an appropriate function space $V^s(\Omega|\mathbb{R}^d)$, which was first introduced in \cite{FKV15}. An optimal choice of exterior data requires the study of trace and extension results, done in \cite{DyKa19}, \cite{BGPR20}, \cite{GrHe22b} and \cite{GrKa23}. Robust estimates like \eqref{eq:main_estimate_theorem} were proved in the latter two articles.
	\end{remark}
	
	The robustness of the constant in \eqref{eq:main_estimate_theorem} in terms of $s$ allows us to retrieve the classical $H^{1/2}$-regularity result for the problem \eqref{eq:dirichlet_laplace} in the limit $s\to 1-$. This follows from a direct application of \autoref{th:main_result} and an appropriate approximation, see \autoref{th:nonlocalTolocal}.
	
	\subsection{Strategy of the Proof}\label{sec:idea_proof}  In the following, we sketch the proof of \autoref{th:main_result} in the special case of the fractional Laplacian. The first step is to prove a weighted Sobolev embedding
	\begin{equation}\label{eq:first_step_proof_sketch}
		\norm{u}_{H^{s/2}(\Omega)}^2 \le c_1\big( \norm{u}_{L^2(\Omega)}^2 + \int\limits_{\Omega} d_x^s \Gamma_s(u)(x) \d x \big),
	\end{equation}
	see \autoref{lem:weighted_sobolev}. Here $\Gamma_s$ is the carr\'{e} du champ related to the fractional Laplacian, see \eqref{eq:carre_du_champ}. It is the symmetric bilinear operator satisfying $(-\Delta)^s (u^2)= 2u (-\Delta)^s(u)- \Gamma_s(u)$. This first step is essentially the fractional analog of \cite[Theorem 4.1]{KenigDirichlet}. Introducing the carr\'{e} du champ is fundamental to gain access to the equation that $u$ satisfies.
	
	In the second step, we introduce a particular auxiliary function $\phi$ to the right-hand side of \eqref{eq:first_step_proof_sketch} in place of $d_x^s$. As such it is essential to analyze the boundary behavior of this function. We choose $\phi$ to be the classical solution to 
	\begin{align*}
		(-\Delta)^s \phi &=1 \quad \text{ in }\Omega,\\
		\hspace{12pt} \phi&= 0 \quad \text{ on }\Omega^c.
	\end{align*}
	This function satisfies $c_2 d_x^{s}\le \phi(x)\le c_3 d_x^s$ for $x\in \Omega$, see \cite[Proposition 2.6.6]{Ros23}. Together with the properties of the carr\'{e} du champ, the second term on the right-hand side of \eqref{eq:first_step_proof_sketch} can be estimated as follows:
	\begin{equation*}
		\int\limits_{\Omega} d_x^s \Gamma_s(u)(x) \d x\le c_2^{-1}\Big( 2\int\limits_{\Omega} \phi(x) u(x) f(x)\d x - \int\limits_{\Omega} \phi(x) (-\Delta)^s(u^2)(x)\d x\Big).
	\end{equation*}
	The first term is easily bounded from above using the Hölder inequality. A nonlocal Gauss-Green formula reveals that the second term equals
	\begin{align*}
		- \int\limits_{\Omega} \phi(x) (-\Delta)^s(u^2)(x)\d x= -\int\limits_{\Omega} u(x)^2 \d x - \int\limits_{\Omega^c} N_s\phi(x) g(x)^2 \d x.
	\end{align*}
	Here $N_s$ is the nonlocal normal derivative, see \eqref{eq:normal_derivative}. The first term is nonpositive and $-N_s\phi$ can be estimated from above by $\tau_s$, see \autoref{lem:normal_derivative_distance}. To show that $u$ is square integrable, we use a similar strategy after introducing the operator to the $L^2$-norm via
	\begin{align*}
		\int\limits_{\Omega} u(x)^2 \d x = \int\limits_{\Omega} (-\Delta)^s \phi(x) u(x)^2 \d x = \int\limits_{\Omega} \phi(x) (-\Delta)^s u(x)^2 \d x- \int\limits_{\Omega^c} N_s\phi(y) u(y)^2 \d x,
	\end{align*}
	see \autoref{lem:L2_bound}. \smallskip
	
	The main difficulty in adapting the proof for the fractional Laplacian to general $2s$-stable operators is to prove the analog of \eqref{eq:first_step_proof_sketch}, \ie to introduce the corresponding carré du champ. This is done in \autoref{sec:ineq_gen_stable}.
	
	\smallskip 
	
	On $C^{1,1}$-domains, it is not necessary to introduce the auxiliary function $\phi$ to prove the bound on $\int_{\Omega} d_x^s \Gamma_s(u)(x)\d x$, because the term $(-\Delta)^s(d_x^s)$ is bounded in that case. Note that, in contrast to the respective term with $\phi$, the sign of $(-\Delta)^s(d_x^s)$ in general domains is unknown and the integral cannot be dropped. For general nondegenerate stable operators $A_s$ the term $A_s (d_x^s)$ is not essentially bounded even in smooth domains $\Omega$, see \cite[Proposition 6.2]{Ros16}. 
	
	In the case of the Laplacian, the solution $\phi_1$ to $-\Delta \phi_1 =1$ in $\Omega$ and $\phi_1=0 $ on $ \partial \Omega$ was also used in \cite{Pe83} to prove a regularity result.

\subsection{Related literature}\label{sec:literature}
Existence theory and Sobolev regularity for weak variational elliptic and parabolic nonlocal problems was studied in \cite{FKV15} and a general class of integro-differential operators with possibly antisymmetric kernels were treated. The article \cite{KiRy23} contains regularity results for distributional solutions to the Dirichlet problem for the fractional Laplacian on $C^{1,1}$-domains. Exterior data and inhomogeneities in certain weighted Sobolev spaces of any order were treated. These weighted spaces were introduced in \cite{Lot00}. A careful study of these spaces shows that their result includes our result for the special case of the fractional Laplacian and $C^{1,1}$-domains. We extend their result to $C^{1, \alpha}$-domains with a robust estimate for $s\to 1-$. The approach in \cite{KiRy23} relies on a potential analysis of the Poisson kernel and Green function close to the boundary. Our approach allows us to consider a larger class of operators and more general domains. The article \cite{KiRy23} is an extension of \cite{CKR22}, where only zero exterior data were treated, and \cite{Gru14}, where $C^\infty$-domains and more regular boundary data were considered.

There is also interest in nonlocal problems with boundary data being only prescribed on $\partial\Omega$, cf.\ \cite{Gru15} and references therein.

An early contribution to boundary regularity of solutions to exterior data problems, like \eqref{eq:main_equation}, is \cite{Bog97}, where the author proved a boundary Harnack principle for the fractional Laplacian on Lipschitz domains. In \cite[Proposition 2.8, Proposition 2.9]{Sil07} Hölder regularity for distributional solutions to $(-\Delta)^su =f$ in $\mathbb{R}^d$ was proved. Interior Hölder regularity for weak $L$-harmonic functions, where $L$ is an integro-differential operator with a kernel comparable to the one of the fractional Laplacian, was studied in \cite[Theorem 1.1]{Kas09}. This was extended to operators with kernels which have critically low singularity and do not allow for standard scaling in \cite[Theorem 3]{KaMi17}. Optimal boundary regularity estimates for distributional solutions to $(-\Delta)^su = f$ in $\Omega$ and $u =0 $ on $\Omega^c$ were proved in \cite[Proposition 1.1]{RoSe14}. There, $\Omega$ is a Lipschitz domain satisfying the uniform outer ball condition. Moreover, higher Hölder regularity up to the boundary of the quotient $u/d_x^s$ on a $C^{1,1}$-domain was studied, see Theorem 1.2 in \cite{RoSe14}. This was extended to all generators of $2s$-stable processes in \cite[Theorem 1.2, Proposition 4.5]{RoSe16}. In \cite[Theorem 1.2, Theorem 1.5]{RoSe17}, boundary Hölder regularity of distributional solutions in $C^{1,\alpha}$-domains was proved. This result allows us to study the boundary behavior of our auxiliary function $\phi$, see \eqref{eq:phi_test_function}. The article \cite{DyKa20} contains Hölder regularity estimates for weak solutions to $Lu=f$ on $\mathbb{R}^d$ for more general integro-differential operators that are defined via measures, which are allowed to be singular. The estimate is robust in the order of differentiability. In \cite[Theorem 1.11]{DyKa20}, the comparability of the energy forms corresponding to singular operators like \eqref{eq:definition_A_s} to the energy form related to the fractional Laplacian was studied. We apply this result in \autoref{lem:comparison_forms}.

Integro-differential operator of variable order and Hölder regularity of solutions to the respective Dirichlet problem in balls was studied in \cite[Theorem 2.2]{BaKa05}. For Hölder regularity estimates of solutions to fully nonlinear nonlocal equations we refer to \cite{CaSi09,CaSi11a,CaSi11b}, and for boundary regularity to \cite{RoSe16}. In \cite{Ser15}, the author proved Hölder regularity of solutions to concave nonlocal fully nonlinear elliptic equations. Therein, the operators are allowed to have non-smooth kernels and do not have to be translation invariant. In \cite{CKW22}, robust Hölder regularity estimates for minimizers of nonlocal functionals with non-standard growth were shown. Interior and boundary regularity for distributional solutions to $Lu=f$ in $C^{1,\alpha}$-domains and zero exterior data are proved in \cite[Theorem 1.1, Theorem 1.2, Corollary 1.3]{DRSV22}, where $L$ is a translation invariant integro-differential operator with a nonsymmetric kernel. Additionally, Theorem 1.5 contains a new integration by parts formula for these kinds of operators. In \cite{KW22}, nonsymmetric non-translation invariant nonlocal operators were treated.

The convergence of fractional Sobolev spaces to classical Sobolev spaces was first established in \cite{BBM01}. In the article \cite{FKV20}, the nonlocal to local convergence in the sense of Mosco for quadratic forms, which appear in the variational study of nonlocal Dirichlet problems, was proved. This result was extended in \cite{Fog20} to a larger class of integro-differential operators. In \cite{Fog20} and \cite{FK22}, the convergence of weak solutions to nonlocal Dirichlet and Neumann problems to weak solutions of local Dirichlet and Neumann problems with nonzero boundary data was studied. The convergence of Neumann problems was extended in \cite{GrHe22b} to a more general class of exterior data. This was achieved by a careful study of the trace spaces related to the fractional-type Sobolev spaces which naturally appear in the setup of the weak formulation to problems like \eqref{eq:main_equation}. For trace and extension results in the $L^p$-setting for all $p \ge 1$ we refer to \cite{GrKa23}.

Green function and Poisson kernel estimates for the fractional Laplacian on $C^{1,1}$-domains were studied in \cite{ChSo98} and \cite{Che99}. In \cite{BoSz05}, a Harnack inequality and on-diagonal Green function estimates for a subclass of $2s$-stable L\'{e}vy processes on balls were proved. This result was extended in \cite{BoSz07} to a more general class of operators. In particular, the article contains a characterization of those operators in the class of $2s$-stable symmetric nonlocal operators that admit a Harnack inequality. For Dirichlet heat kernel estimates for cylindrical stable processes on $C^{1,1}$-domains we refer to \cite{CHZ23}.

Finally, we give a short overview of results concerning the Dirichlet problem for second-order operators like \eqref{eq:dirichlet_laplace}. To our knowledge the first boundary regularity result with $L^2$-boundary data was achieved in \cite{Mi76} who studied the problem in $C^2$-domains. This result was generalized in \cite{CT83} to a larger class of operators. The case of $C^{1,\alpha}$-domains and boundary data in $L^p$, $1<p<\infty$, was treated in \cite{Pe83}. In \cite{Li91}, the assumption on the boundary $\partial \Omega$  was reduced to $C^{1,\text{Dini}}$. In the influential articles \cite{Ken94, KenigDirichlet} the $H^{1/2}(\Omega)$-regularity result was proved in Lipschitz domains. The proof relies on fine estimates of the harmonic measure from \cite{DahlbergEstimatesHarmonicMeasure, DahlbergOnThePoissonIntegral} and weighted norm inequalities for the nontangential maximal function from \cite{Dah80}. 

\subsection{Outline}

In \autoref{sec:prelims}, we introduce the notation used throughout this work as well as the auxiliary function $\phi$. In \autoref{sec:ineq_gen_stable}, we prove the key estimates to relate the carré du champ, associated with the class of $2s$-stable operators, with the $H^{s\slash 2}$-seminorm. \autoref{sec:weighted_sobolev} contains a robust weighted Sobolev inequality which yields the $H^{s/2}$-regularity of distributional solutions. The proof of \autoref{th:main_result} is given in \autoref{sec:proof}. As an application, we retrieve the $H^{1/2}$-regularity of solutions to the local Dirichlet problem with boundary data in $L^2(\partial \Omega)$ by a nonlocal to local convergence of solutions in \autoref{sec:nonlocaltolocal}.
 
\subsection*{Acknowledgments}
Financial support by the German Research Foundation (GRK 2235 - 282638148) is gratefully acknowledged. We thank Moritz Kassmann for helpful discussions and Solveig Hepp for valuable comments on the manuscript. 

\section{Preliminaries}\label{sec:prelims}
In this section, we introduce the notation we use through this work. Additionally, we define basic objects like the carr\'{e} du champ and the nonlocal normal derivative. Lastly, we provide the existence of the auxiliary function $\phi$. 

\smallskip

Let us denote $d_x := \distw{x}{\partial\Omega}$ and $\Omega_\eps := \{ x\in \Omega\,|\, d_x>\eps \}$. We use $a \wedge b$ for $\min\{a,b\}$ as well as $a \vee b$ for $\max\{a,b\}$. With $a^+$ and $a^-$ we denote the positive respectively the negative part of $a$. If not mentioned differently then $\Omega$ is always a bounded $C^{1,\alpha}$-domain which means there exists a localization radius $r_0>0$ such that for any $z\in \partial \Omega$ there exists a rotation and translation $T_z:\R^d\to \R^d$ with $T_z(z)=0$ and a $C^{1,\alpha}$-continuous map $\phi_z:\R^{d-1}\to \R$ such that $T_z(\Omega\cap B_{r_0}(z))= \{ (y',y_d)\in B_{r_0}(0)\,|\, \phi_z(y')> y_d \}$. Since $\partial \Omega$ is compact, we find a uniform constant $L=L(\Omega)\ge 1$ such that $\norm{\phi}_{C^{1,\alpha}}\le L$. In the proofs we use small $c_1, c_2, c_3 , \dots$ to mark different constants and in the statements of the theorems we use the capital letter $C$. 

For $s\in (0,1)$ we write $H^{s}(\Omega)$ for the classical $L^2$-based Sobolev-Slobodeckij space on a domain $\Omega$. It is endowed with the norm
\begin{align*}
	\norm{u}_{H^s(\Omega)} := \Big(\norm{u}_{L^2(\Omega)}^2 + (1-s)\int_{\Omega} \int_{\Omega} \frac{\abs{u(x)-u(y)}^2}{\abs{x-y}^{d+2s}} \d x \d y \Big)^{1/2}.
\end{align*}
With $L^p(\Omega;w(x))$ we denote the $L^p$-space where integration is with respect to the weighted Lebesgue measure $w(x)\d x$. For all $\alpha>0$ we write $C^\alpha(\overline{\Omega}) = C^{\lfloor \alpha \rfloor,\alpha-\lfloor \alpha \rfloor }(\overline{\Omega})$ for the space of all Hölder continuous function on $\overline{\Omega}$ equipped with the norm $\norm{u}_{C^{\alpha}}= \norm{u}_{C^{\lfloor \alpha \rfloor}(\Omega)}+  \sup_{\abs{\beta}= \lfloor \alpha \rfloor } [\partial^\beta u]_{C^{\alpha-\lfloor \alpha \rfloor }(\Omega)}$. Here we write for $0< s <1$ \begin{equation*}
	[u]_{C^s(\Omega)}= \sup\limits_{x,y\in \Omega} \frac{\abs{u(x)-u(y)}}{\abs{x-y}^s}.
\end{equation*}
The set of locally Hölder continuous functions is denoted by $C_{\loc}^\alpha(\Omega)$. Furthermore, we define the carr\'{e} du champ operator related to $A_s$ as
\begin{equation}\label{eq:carre_du_champ}
	\Gamma_s(u)(x) := \int_{\mathbb{R}^d} (u(x)-u(x+h))^2 d\nu_s(h) \quad\text{for $x\in \mathbb R^d$.}
\end{equation}
The nonlocal normal derivative related to $A_s$ is defined via
	\begin{equation}\label{eq:normal_derivative}
		N_s(u) (x) := (1-s)\int_{\mathbb{R}} \int_{\mathbb S^{d-1}} \1_{\Omega}(x+r\theta) \frac{u(x)-u(x+r\theta)}{\abs{r}^{1+2s}} \d \mu(\theta) \d r \quad \text{for $x\in \overline{\Omega}^c$.}	
	\end{equation} 

This nonlocal normal derivative appeared first in \cite{DROV17}. We also refer to the work \cite{DGLZ12} who studied a similar operator in the context of peridynamics.

\begin{lemma}\label{lem:normal_derivative_distance}
	Let $\Omega\subset\mathbb R^d$ be a bounded open set. There exists a constant $C=C(\Omega, \Lambda)\ge 1$ such that for any $s\in (0,1)$
	\begin{align*}
		-N_s(d^s)(x)\le \frac{C}{s}\, \tau_s(x)\quad \text{for all $x\in \overline{\Omega}^c$.}	
	\end{align*}
\end{lemma}
\begin{proof}
	For $x\in \overline{\Omega}^c$ such that $d_x\ge 1$ the claim follows from a calculation analogous to \cite[Lemma 2.4, equation (2.8)]{GrHe22a}. In particular, \cite[equation (2.8)]{GrHe22a} shows that $-N_s(d^s)(x)=0$ if $x\notin \supp \nu_s^\star$. Thus, it is sufficient to consider $x\in \overline{\Omega}^c\cap \supp \nu_s^\star$ satisfying $d_x<1$. In this case we calculate:
	\begin{align}\label{eq:normal_derivative_test_function}
		\begin{split}
			-N_s(d^s)(x)&\le (1-s)\int\limits_{\R^d}d_{x+h}^s \1_{\Omega}(x+h)\nu_s(\d h)\le(1-s) \int\limits_{\abs{r}>d_x}\int\limits_{\mathbb S^{d-1}} \frac{1}{\abs{r}^{1+s}}\mu(\d \theta)\d r\\
			&\le \frac{2 \mu(\mathbb S^{d-1})}{s} (1-s)d_x^{-s} \le \frac{2\Lambda}{s} \tau_s(x).
		\end{split}
	\end{align}	
\end{proof}
The proof of our main result relies on a specific auxiliary function and a careful analysis thereof near the boundary of $\Omega$. This is done in the following lemma. 

\begin{lemma}\label{lem:phi_test_function}
	Let $\Omega \subset \R^d$ be a $C^{1,\alpha}$-domain, $\alpha \in (0,1)$ and $s_\star \in (0,1)$. For any $s\in (s_\star, 1)$ there exists a classical solution $\phi \in C^{s}(\R^d)\cap C_{\loc}^{3s}(\Omega)$ to the problem 
	\begin{align}\label{eq:phi_test_function}
	\begin{split}
	A_s \phi &= 1 \quad \text{in }\Omega,\\
	\phi &= 0 \quad \text{on }\Omega^c.
	\end{split}
	\end{align}
	Furthermore, there exists a constant $C=C(d,\Omega, \alpha, s_\star, \lambda,\Lambda)\ge 1$ such that 
	\begin{equation}\label{eq:boundary_behavior_test_function}
	C^{-1} d_x^s \le \phi(x)\le C d_x^s \quad \text{for all $x\in \overline{\Omega}$}.
	\end{equation}
\end{lemma}
The function $\phi$ is the expected value of the first exit time from the domain $\Omega$ of the stochastic process that is generated by $A_s$.
\begin{proof}
	The existence of a weak solution $\phi$ follows from Lax-Milgram, see \eg \cite[Proposition 2.11]{GrHe22a}. This together with a nonlocal Gauss-Green formula and the regularity theory, see \cite{RoSe17} and \cite[Proposition 2.6.9]{Ros23}, reveals that $\phi\in C^s(\R^d)\cap C^{3s}_{\loc}(\Omega)$ is a classical solution. The lower bound in \eqref{eq:boundary_behavior_test_function} follows from the Hopf lemma \cite[Proposition 2.6.6]{Ros23} and the upper bound was proved in \cite[(2.6.3)]{Ros23}. The dependencies of the constant $C$ follow from carefully following the constants in the proofs. For the convenience of the reader we provide a meta analysis of the proofs of Proposition 2.6.6 and (2.6.3) from \cite{Ros23} in \autoref{rem:ros_oton_constant}.
\end{proof}

\section{Functional inequalities}\label{sec:functional_inequalities}
In this section, we provide most of the important estimates for the proof of \autoref{th:main_result}. The necessary ingredients to lift the proof from the fractional Laplacian, as described in \autoref{sec:idea_proof}, to general $2s$-stable operators are stated in \autoref{sec:ineq_gen_stable}. In particular, the necessary bound to introduce the carr\'{e} du champ of $A_s$ and gain access to \eqref{eq:main_equation} is done in \autoref{lem:comparison_forms}. The last part of this section concerns a weighted Sobolev inequality which already proves the first inequality in \eqref{eq:main_estimate_theorem}. Our main contribution there is the robustness of the constant. 

\subsection{Inequalities for general $2s$-stable operators}\label{sec:ineq_gen_stable}

\begin{lemma}\label{lem:comparison_forms}
	Let $\Omega\subset \R^d$ be a bounded open set, $s_\star \in (0, 1)$, $\beta>0$ and $\mu$ a finite measure on $\mathbb S^{d-1}$ that satisfies \eqref{eq:ellip_conditions}. There exists a constant $C=C(d,\Omega, s_\star, \lambda, \Lambda)>0$ such that for any $s\in (s_\star, 1)$
	\begin{align}\label{eq:prop_caree_comparison}
		\begin{split}
			&\iint\limits_{\Omega\,\Omega}  \big(d_x\wedge d_y\big)^{\beta} \frac{\abs{u(x)-u(y)}^2}{\abs{x-y}^{d+2s}}\d y \d x \le C\, \int\limits_{\Omega} \frac{\abs{u(x)}^2}{d_x^{2s-\beta}}\d x\\
			&\qquad\qquad+ C\,5^\beta \int\limits_{\Omega}  \int\limits_{\R} \int\limits_{\mathbb S^{d-1}} \big(d_x\wedge d_{x+r\theta}\big)^\beta \frac{\abs{u(x)-u(x+r\theta)}^2}{\abs{r}^{1+2s}}\1_{\Omega}(x+r\theta) \mu(\d \theta)\d r \d x
		\end{split}
	\end{align}
	for any $u:\R^d\to \R$ measurable.
\end{lemma}
The weighted $L^2$-term on the right-hand side of \eqref{eq:prop_caree_comparison} can be dropped if one proves a robust Poincaré inequality for the energy corresponding to $A_s$. Instead, we treat this term using \autoref{prop:distance_estimate} to receive the $L^2(\Omega^c; \tau_s)$-norm of the exterior data, see \autoref{cor:comparison_forms}. 
\begin{proof}
	We divide the proof into three steps. 
	
	\textit{Step 1:} We prove that there exists a constant $c_1=c_1(d, s_\star)>0$ such that 
	\begin{equation}\label{eq:step_1_energy_comparison}
			\int\limits_{\Omega}  \int\limits_{\Omega\cap B_{d_x/8}(x)^c}  \big(d_x\wedge d_y\big)^\beta \frac{\abs{u(x)-u(y)}^2}{\abs{x-y}^{d+2s}}\d y \d x \le c_1 \int\limits_{\Omega} \frac{\abs{u(x)}^2}{d_x^{2s-\beta}}\d x.
	\end{equation}
	After an application of the triangle inequality $\abs{u(x)-u(y)}^2 \le 2 (\abs{u(x)}^2+\abs{u(y)}^2)$ we are left with two integrals to estimate. The first one is:
	\begin{align*}
			\int\limits_{\Omega} \abs{u(x)}^2 \int\limits_{\Omega\cap B_{d_x/8}(x)^c}  \big(d_x\wedge d_y\big)^\beta \frac{1}{\abs{x-y}^{d+2s}}\d y \d x \le \frac{\omega_{d-1}}{2s}\int\limits_{\Omega}  d_x^\beta \abs{u(x)}^2 \Big( \frac{8}{d_x} \Big)^{2s} \d x.
	\end{align*}
	The second one needs a little bit more care. Let $x,y\in \Omega$ such that $\abs{x-y}>d_x/8$. If $8d_y\le 9d_x$, then $\abs{x-y}>d_x/8\ge d_y/9$. If $8d_y>9d_x$, then $d_y\le \abs{y-z}\le \abs{y-x}+\abs{x-z}$ for any $z\in \partial \Omega$. Thus, $\abs{y-x}\ge d_y-d_x\ge d_y/9$. After this observation, the second term can be treated analogous to the first:
	\begin{align*}
		\int\limits_{\Omega} \abs{u(y)}^2 \int\limits_{\Omega}  \big(d_x\wedge d_y\big)^\beta \frac{\1_{B_{d_x/8}(x)^c}(y)}{\abs{x-y}^{d+2s}}\d x \d y &\le \int\limits_{\Omega} d_y^\beta \abs{u(y)}^2 \int\limits_{\Omega\cap B_{d_y/9}(y)^c} \frac{1}{\abs{x-y}^{d+2s}}\d x \d y\\
		&\le \frac{\omega_{d-1}}{2s} \int\limits_{\Omega} d_y^\beta \abs{u(y)}^2 \Big( \frac{9}{d_y} \Big)^{2s} \d y.
	\end{align*}
	
	\textit{Step 2:} We claim that there exists a constant $c_2=c_2(d, s_\star, \lambda, \Lambda)>0$ such that for any ball $B_r(x_0)\subset \R^d$ we have:
	\begin{align}\label{eq:step_2}
		\int\limits_{B_r(x_0)} \int\limits_{B_r(x_0)} \frac{\abs{u(x)-u(y)}^2}{\abs{x-y}^{d+2s}}\d y \d x \le c_2 \int\limits_{B_{r}(x_0)} \int\limits_{\R} \int\limits_{\mathbb S^{d-1}} \frac{\abs{u(x)-u(x+t\theta)}^2}{\abs{t}^{1+2s}}\1_{B_r(x_0)}(x+t\theta) \mu(\d \theta) \d t \d x.
	\end{align}
	This result follows from \cite[Theorem 1.11]{DyKa20} together with a scaling argument. Note that \eqref{eq:ellip_conditions} implies the nondegeneracy assumption $\lin \supp \mu = \R^d$ in \cite[Theorem 1.11]{DyKa20}. \smallskip
	
	\textit{Step 3:} We localize the problem and apply step 2. Let $\cW(\Omega)=\cW$ be a Whitney covering of the open set $\Omega$ by balls, see \eg  \cite[Theorem 3.2]{CoWe77}. We denote for $B\in \cW$ by $x_B$ the center, by $r_B$ the radius of the ball $B$ and write $B^\star = 7/4\,B$ for the ball with $7/4$-times the radius of $B$ but the same center. There exists a number $N=N(d,\Omega)\in \N$ such that 
	\begin{align}\label{eq:whitney_balls}
		\begin{split}
			\Omega = \bigcup\limits_{B\in \cW} B, \qquad 2 r_B\le  \distw{x_B}{\Omega^c}\le 4 r_B,\\
			\sum\limits_{B\in \cW} \1_{B^\star}(x) \le N \quad \text{for all }x\in \Omega.
		\end{split}
	\end{align}
 \smallskip 
	
	By step 1, it suffices to consider the integral over $(x,y)\in \Omega\times \Omega$ such that $\abs{x-y}<d_x/8$. For such $(x,y)$ and $B\in \cW$ such that $x\in B$ it follows that $\abs{x_B-y}\le \abs{x_B-x}+ \abs{x-y}\le r_B + d_x/8\le r_B + 1/8 ( r_B + \distw{x_B}{\Omega^c})\le (1+5/8)r_B < 7/4\, r_B$. We use the Whitney covering to find:
	\begin{align*}
		\int\limits_{\Omega}  \int\limits_{\Omega\cap B_{d_x/8}(x)}  \big(d_x\wedge d_y\big)^\beta \frac{\abs{u(x)-u(y)}^2}{\abs{x-y}^{d+2s}}\d y \d x &\le \sum\limits_{B\in \cW} \,\int\limits_B  \int\limits_{\Omega\cap B_{d_x/8}(x)}  \big(d_x\wedge d_y\big)^\beta \frac{\abs{u(x)-u(y)}^2}{\abs{x-y}^{d+2s}}\d y \d x \\
		&\le \sum\limits_{B\in \cW} (5r_B)^\beta\,\int\limits_B  \int\limits_{\Omega\cap B_{d_x/8}(x)}  \frac{\abs{u(x)-u(y)}^2}{\abs{x-y}^{d+2s}}\d y \d x\\
		&\le \sum\limits_{B\in \cW} (5r_B)^\beta\,\int\limits_B  \int\limits_{B^\star}  \frac{\abs{u(x)-u(y)}^2}{\abs{x-y}^{d+2s}}\d y \d x.
	\end{align*}
	Using \eqref{eq:step_2} for each ball $B^\star$, we estimate the right-hand side of the previous inequality by:
	\begin{equation*}
		\sum\limits_{B\in \cW} (5r_B)^\beta\iint\limits_{B\,B^\star}  \frac{\abs{u(x)-u(y)}^2}{\abs{x-y}^{d+2s}}\d y \d x\le 5^\beta c_2 \sum\limits_{B\in \cW} r_B^\beta \int\limits_{B^\star} \int\limits_{\R} \int\limits_{\mathbb S^{d-1}} \frac{\abs{u(x)-u(x+t\theta)}^2}{\abs{t}^{1+2s}}\1_{B^\star}(x+t\theta) \mu(\d \theta) \d t \d x.
	\end{equation*}
	Finally, the observation that for $x\in B^\star$ the distance of $x$ to the boundary $\partial \Omega$ can be estimated by $d_x\ge d_{x_B}-r_B\ge r_B$ together with the finite overlap property from \eqref{eq:whitney_balls} completes the proof.
\end{proof}

\smallskip

\begin{proposition}\label{prop:distance_estimate}
	Let $\Omega\subset \R^d$ be a bounded $C^{1,\alpha}$-domain with $\alpha\in (0,1)$, $s_\star\in (0, 1)$ and $\mu$ a finite measure on $\mathbb S^{d-1}$ that satisfies the nondegeneracy assumption \eqref{eq:ellip_conditions}. There exist two constants $C=C(\Omega,\alpha, s_\star, \lambda, \Lambda)>0$, $a=a(s_\star, \lambda, \Lambda)>1$ and a radius $\rho=\rho(\Omega,\alpha, s_\star, \lambda, \Lambda)>0$ such that for any $s\in (s_\star, 1)$
	\begin{align*}
		d_x^{-2s}\le C\int\limits_{\abs{r}<ad_x} \int\limits_{\mathbb S^{d-1}} \frac{\1_{\Omega^c}(x+r\theta)}{\abs{r}^{1+2s}}\mu(\d \theta) \d r \quad \text{for all $x\in \Omega$ such that $d_x\le \rho$.}
	\end{align*}
\end{proposition}
We actually do not need $\Omega$ to be bounded nor connected in this lemma. What we really need is a uniform bound on the $C^{1,\alpha}$-norm of the maps describing the boundary of $\Omega$ locally. For the special case of cylindrical stable processes and when $\Omega$ is a ball this result is proven in \cite[Lemma 2.3]{CHZ23}. Moreover, the result for the fractional Laplacian is straight forward and also true for Lipschitz domains, see e.g.\ \cite{GrKa23}. In contrast, \autoref{prop:distance_estimate} is not true in general Lipschitz domains and general $\mu$ even under the assumption \eqref{eq:ellip_conditions}. This is due to the fact that in a domain with a corner the support of $\nu_s$ translated by $x\in \Omega$ may not hit the boundary close to $x$. An example for this phenomenon is $\Omega= (-2,2)^2 \setminus [0,1)^2\subset \R^2$, $\mu= \delta_{e_1}+\delta_{e_2}$, where $e_i$ is the $i$-th coordinate vector, and $x=(-t,-t)$ for $t\in (0,1/2)$. 

\begin{proof}
	Since $\Omega$ is a bounded $C^{1,\alpha}$-domain, we find a localization radius $r_0>0$ such that for any $z\in \partial \Omega$ there exists a rotation and translation $T_z:\R^d\to \R^d$ with $T_z(z)=0$ and a $C^{1,\alpha}$-continuous map $\phi_z:\R^{d-1}\to \R$ such that $T_z(\Omega\cap B_{r_0}(z))= \{ (y',y_d)\in B_{r_0}(0)\,|\, \phi_z(y')> y_d \}$. Since $\partial \Omega$ is compact, we find a uniform constant $L=L(\Omega)\ge 1$ such that $\norm{\phi}_{C^{1,\alpha}}\le L$. 
	
	\smallskip
	
	We set 
	\begin{equation*}
		a:= \Big( \frac{8 \Lambda}{\lambda}+1 \Big)^{1/(2s_\star)}+1 \qquad\text{and}\qquad \rho:=  \frac{1}{a\,L^{1/\alpha} (a-1)^{1/\alpha}  }\wedge \frac{r_0}{2(a+1)}\wedge  \frac{1}{a}\Big( \frac{r_0}{\sqrt{1+L^2} }\Big)^{1/(1+\alpha)}. 
	\end{equation*}
	Now let $x=(x',x_d)\in \Omega$ such that $d_x<\rho$. Let $z_x\in \partial \Omega$ be a minimizer of $x$ to the boundary $\partial \Omega$. Since $\partial \Omega$ is continuously differentiable, the vector $(z_x-x)/\abs{z_x-x}$ is the outer normal vector in $z_x$, which we call for short $n_x$. Without loss of generality we assume that $T_{z_x}=\Id$ and $x=(0,-d_x)$. From now on we simply write $\phi=\phi_{z_x}$. This is due to $A_s$ being translation invariant and the ellipticity assumption \eqref{eq:ellip_conditions} is invariant under rotations. Since $z_x=0$ minimizes the distance of $x=(0,-d_x)$ to the graph of $\phi$, the gradient of $\phi$ is zero in $z_x$, \ie $\nabla \phi(0)=0$. Otherwise we easily find $y\in \partial \Omega$ which is closer to $x$ than $z_x$ by a Taylor expansion of $\phi$. \smallskip
	
	A simple calculation yields
	\begin{align}\label{eq:calculation_lemma_distance_estimate}
		\begin{split}
			\lambda \,\frac{\big( (a-1)^{-2s}-a^{-2s} \big)}{2s}  d_x^{-2s}&=\lambda\, \int\limits_{(a-1)d_x}^{a d_x} \frac{1}{r^{1+2s}}\d r\le \int\limits_{(a-1)d_x}^{a d_x}\int\limits_{\mathbb S^{d-1}} \frac{\abs{n_x\cdot \theta}^{2s}}{r^{1+2s}}\mu(\d \theta)\,\d r\\
			&\le \int\limits_{(a-1)d_x}^{a d_x}\int\limits_{\mathbb S^{d-1}} \frac{\abs{n_x\cdot \theta}^{2s}}{r^{1+2s}}\big(\1_{\Omega^c}(x+r\theta)\vee \1_{\Omega^c}(x-r\theta)\big)\mu(\d \theta)\,\d r\\
			&\quad +  \int\limits_{(a-1)d_x}^{a d_x}\int\limits_{\mathbb S^{d-1}} \frac{\abs{n_x\cdot \theta}^{2s}}{r^{1+2s}}\big(\1_{\Omega}(x+r\theta)\wedge \1_{\Omega}(x-r\theta)\big)\mu(\d \theta)\,\d r\\
			&=: (\text{I})+ (\text{II}).
		\end{split}
	\end{align}
	We will show that the term $(\text{II})$ is smaller than one half of the left-hand side of \eqref{eq:calculation_lemma_distance_estimate} such that we can absorb the term. To achieve this, we need to find an upper bound of $\abs{n_x\cdot \theta}$ for $\theta=(\theta',\theta_d)\in \mathbb S^{d-1}$ and any $r\in ((a-1)d_x, ad_x)$ such that $x+r\theta$ and $x-r\theta$ are in $\Omega$. We fix $x+r\theta$ with this property. Note that $x+r\theta \in B_{r_0}(z_x)$ since $\rho\le r_0/(2(a+1))$. We distinguish two cases. \smallskip
	
	\textit{Case 1:} If $r\theta_d\ge 0$, then we consider the vector $y$ that lies above $x+r\theta$ on the boundary $\partial \Omega$, \ie $y:= (x'+r\theta', \phi(x'+r\theta'))=(r\theta', \phi(r\theta'))$. By the last term in the choice of $\rho$, the vector $y$ is in $B_{r_0}(z_x)$. Surely, since $x+r\theta$ is in $\Omega$, we have $x_d+r\theta_d \le \phi(r\theta')$. Thus, we estimate:
	\begin{align*}
		\abs{ x+r\theta  - (x'+r\theta', x_d) }&= \abs{r\theta_d}= r\theta_d \le \phi(r\theta')-x_d\le L \abs{r\theta'}^{1+\alpha} + d_x\\
		&\le  \big( L a^\alpha d_x^\alpha + (1-s)^{-1}\big) \, r\le 2/(a-1)\, r.
	\end{align*}
	Here we used that $\phi(0) = \nabla \phi(0)=0$ to estimate $\phi(r\theta')=\abs{\phi(r\theta')-\phi(0) - \nabla\phi(0)\cdot(r\theta')}$ and $d_x<\rho$ in the last inequality. This yields an upper bound of $\abs{n_x\cdot \theta}$ as follows:
	\begin{equation}\label{eq:inner_product_bound}
		\abs{n_x\cdot \Big( \frac{(x+r\theta)-x}{\abs{(x+r\theta)-x}} \Big)}= \abs{\cos(\beta)}= \frac{\abs{ x+r\theta  - (x'+r\theta', x_d) }}{\abs{(x+r\theta)-x}}\le \frac{2}{a-1}.
	\end{equation}
	Here $\beta$ is the angle between $n$ and $(y-x)/\abs{y-x}$, see \autoref{fig:geometry_lemma_A_10}.\smallskip
	
	\textit{Case 2:} If $r\theta_d<0$, then we use the property $x-r\theta\in \Omega$ and apply case 1 to $x+r(-\theta)$. By symmetry, this yields the same bound \eqref{eq:inner_product_bound}. \smallskip
	
	\begin{figure}[!ht]
		\centering
	\begin{tikzpicture}[y=1cm, x=1cm, yscale=0.65,xscale=0.65, inner sep=0pt, outer sep=0pt]

	\path[draw=white, pattern color=gray!100,pattern=north east lines]
	(7.2757,18.5158) .. controls (7.0419, 18.5621) and (6.8020, 18.6454) ..
	(6.5656, 18.6256) .. controls (6.3299, 18.6058) and (6.0885, 18.6656) ..
	(5.8481, 18.6193) .. controls (5.6198, 18.5753) and (5.4004, 18.5361) ..
	(5.1714, 18.4788) .. controls (4.9249, 18.4171) and (4.6696, 18.3568) ..
	(4.4526, 18.2252) .. controls (4.1852, 18.0630) and (3.9225, 17.8910) ..
	(3.6343, 17.7668) .. controls (3.4082, 17.6693) and (3.3842, 17.3596) ..
	(3.3612, 17.1372) .. controls (3.3401, 16.9325) and (3.2143, 16.7447) ..
	(3.2129, 16.5426) .. controls (3.2112, 16.2940) and (3.1903, 16.0478) ..
	(3.1234, 15.8089) .. controls (3.0724, 15.6267) and (3.0143, 15.4265) ..
	(3.3247, 15.4310) .. controls (3.5353, 15.4341) and (3.6745, 15.4304) ..
	(3.8834, 15.4652) .. controls (4.1337, 15.5069) and (4.3868, 15.5405) ..
	(4.6360, 15.5111) .. controls (4.9759, 15.4710) and (5.3161, 15.4251) ..
	(5.6463, 15.3282) .. controls (5.7547, 15.2964) and (5.9084, 15.1483) ..
	(5.8899, 15.3813) .. controls (5.8652, 15.6906) and (5.9140, 15.9983) ..
	(5.9807, 16.2979) .. controls (6.0514, 16.6157) and (6.2045, 16.9069) ..
	(6.3237, 17.2105) .. controls (6.4029, 17.4123) and (6.5422, 17.5764) ..
	(6.6622, 17.7541) .. controls (6.7723, 17.9172) and (6.8945, 18.0692) ..
	(7.0206, 18.2276) .. controls (7.1030, 18.3310) and (7.1959, 18.4125) ..
	(7.2757, 18.5158) -- cycle;

	\path[draw=white, pattern color=gray!100,pattern=north east lines]
	(13.6781,11.9292) .. controls (13.9119, 11.8829) and (14.1517, 11.7997) ..
	(14.3880, 11.8195) .. controls (14.6237, 11.8392) and (14.8650, 11.7795) ..
	(15.1053, 11.8258) .. controls (15.3335, 11.8698) and (15.5529, 11.9088) ..
	(15.7818, 11.9662) .. controls (16.0282, 12.0279) and (16.2827, 12.0894) ..
	(16.5004, 12.2197) .. controls (16.7188, 12.3503) and (16.9599, 12.4413) ..
	(17.1736, 12.5799) .. controls (17.3686, 12.7064) and (17.4196, 12.9677) ..
	(17.5033, 13.1755) .. controls (17.5786, 13.3623) and (17.5446, 13.5663) ..
	(17.6215, 13.7569) .. controls (17.7121, 13.9814) and (17.7784, 14.2283) ..
	(17.7328, 14.4629) .. controls (17.6982, 14.6408) and (17.9065, 15.0537) ..
	(17.7311, 15.0253) .. controls (17.5173, 14.9908) and (17.3597, 14.9840) ..
	(17.1608, 14.9805) .. controls (16.9119, 14.9762) and (16.6672, 14.9748) ..
	(16.4173, 14.9481) .. controls (16.1569, 14.9201) and (15.8944, 15.0312) ..
	(15.6332, 15.0349) .. controls (15.4252, 15.0377) and (15.0843, 15.4497) ..
	(15.0820, 15.2551) .. controls (15.0784, 14.9367) and (15.0205, 14.6241) ..
	(14.9932, 14.3059) .. controls (14.9671, 14.0028) and (14.8715, 13.7055) ..
	(14.7326, 13.4374) .. controls (14.6377, 13.2541) and (14.5502, 13.0714) ..
	(14.4340, 12.9024) .. controls (14.3135, 12.7270) and (14.1979, 12.5636) ..
	(14.0785, 12.3918) .. controls (13.9707, 12.2367) and (13.8171, 12.1063) ..
	(13.7045, 11.9536) .. controls (13.7045, 11.9536) and (13.6782, 11.9292) ..
	(13.6782, 11.9292) .. controls (13.6782, 11.9292) and (13.6781, 11.9292) ..
	(13.6781, 11.9292) -- cycle;

	\path[draw=black,fill=black,line join=bevel,line width=0.102cm] (10.5230,
	16.6453) ellipse (0.1021cm and 0.1015cm);

	\path[draw=black,fill=black,line join=bevel,line width=0.102cm] (4.9940,
	17.6987) ellipse (0.1021cm and 0.1015cm);

	\path[draw=black,fill=black,line join=bevel,line width=0.102cm] (10.5230,
	15.1654) ellipse (0.1021cm and 0.1015cm);

	\path[draw=black,line join=bevel,line width=0.060cm] (10.4684, 15.2596) ellipse
	(4.5948cm and 4.5688cm);

	\path[draw=black,line join=bevel,line width=0.060cm] (10.4353, 15.2995) ellipse
	(7.3350cm and 7.2936cm);

	\path[draw=c595959,fill=c5a5a5a,dash pattern=on 0.040cm off 0.120cm,even odd
	rule,draw opacity=0.993,line join=bevel,line width=0.040cm,miter
	limit=3.80,dash phase=0.027cm] (4.9986,17.6863) -- (10.5269,15.1602);

	\path[draw=c595959,fill=c5a5a5a,even odd rule,draw opacity=0.993,line
	join=bevel,line width=0.000cm,miter limit=3.80,dash phase=0.027cm]
	(4.1,16.9) node[above right] (text7276){$x+r\theta$};

	\path[draw=black,line width=0.081cm] (1.2653,14.9948) .. controls (1.6767,
	15.9072) and (2.3638, 16.7454) .. (3.0981, 17.3623) .. controls (4.4670,
	18.5126) and (6.2830, 19.0298) .. (7.6756, 18.3783) .. controls (8.5009,
	17.9922) and (8.9116, 17.0556) .. (9.8447, 16.7595) .. controls (10.8184,
	16.4506) and (11.8822, 16.7891) .. (12.9106, 16.8511) .. controls (13.7129,
	16.8995) and (14.1049, 16.1447) .. (14.6137, 15.6459) .. controls (15.3576,
	14.9168) and (16.4357, 14.8722) .. (17.4453, 15.0145) .. controls (18.3721,
	15.1451) and (18.6826, 14.2275) .. (19.0728, 13.5551) .. controls (19.2496,
	13.2504) and (19.4871, 13.0385) .. (19.7896, 12.9104);

	\path[draw=black,dash pattern=on 0.324cm off 0.162cm on 0.081cm off 0.162cm,line
	width=0.081cm] (19.7462,15.4688) .. controls (19.3360, 14.5563) and (18.6509,
	13.7181) .. (17.9189, 13.1011) .. controls (16.5542, 11.9508) and (14.7436,
	11.4335) .. (13.3553, 12.0851) .. controls (12.5324, 12.4712) and (12.1229,
	13.4078) .. (11.1927, 13.7039) .. controls (10.2220, 14.0129) and (9.1613,
	13.6744) .. (8.1361, 13.6123) .. controls (7.3361, 13.5639) and (6.9453,
	14.3188) .. (6.4380, 14.8176) .. controls (5.6964, 15.5467) and (4.6215,
	15.5913) .. (3.6151, 15.4491) .. controls (2.6910, 15.3185) and (2.3815,
	16.2361) .. (1.9924, 16.9086) .. controls (1.8161, 17.2133) and (1.5794,
	17.4251) .. (1.2778, 17.5533);

	\path[draw=black,opacity=1,even odd rule,line join=bevel,line width=0.000cm]
	(0.7814,14.4) node[above right] (text428){$\partial \Omega$};

	\path[xscale=0.959793,yscale=1.04189,draw=black,fill=black,opacity=1,even
	odd rule,line join=bevel,line width=0.000cm] (19.4759,12.4) node[above
	right] (text1460){$\Omega$};

	\path[xscale=0.895011,yscale=1.1173,draw=black,fill=black,opacity=1,even odd
	rule,line join=bevel,line width=0.000cm] (21.5480,12.14791) node[above right]
	(text2190){$\Omega^c$};

	\path[fill=black,opacity=1,even odd rule,line join=bevel,line width=0.000cm]
	(0.7820,17.6946) node[above right] (text2478){$\partial \widetilde{\Omega}$};

	\path[draw=c595959,fill=c595959,dash pattern=on 0.040cm off 0.120cm,even odd
	rule,draw opacity=0.993,line join=bevel,line width=0.040cm,dash phase=0.040cm]
	(10.5305,16.6245) -- (10.5278,15.1265);

	\path[draw=black,fill=black,even odd rule,line join=bevel,line
	width=0.000cm,miter limit=3.80] (10.6557,15.8687) node[above right, scale=0.8]
	(text4863){$d_x$};

	\path[draw=c595959,fill=c595959,dash pattern=on 0.040cm off 0.120cm,even odd
	rule,draw opacity=0.993,line join=bevel,line width=0.040cm,miter limit=3.80]
	(10.5147,15.1454) -- (14.2742,17.7854);

	\path[draw=c595959,fill=c595959,dash pattern=on 0.040cm off 0.120cm,even odd
	rule,draw opacity=0.993,line join=bevel,line width=0.040cm,miter limit=3.80]
	(10.5029,15.0900) -- (17.2329,18.0060);

	\path[draw=c595959,fill=c595959,even odd rule,draw opacity=0.993,line
	join=bevel,line width=0.000cm,miter limit=3.80] (15.70,17.480) node[above
	right, rotate= 21.3701, scale =0.8] (text5564){$a d_x$};

	\path[fill=c595959,even odd rule,draw opacity=0.993,line join=bevel,line
	width=0.000cm,miter limit=3.80] (12.6157,16.785) node[above right, rotate= 35.0001, scale =0.8]
	(text5670){$(a-1) d_x$};

	\path[fill=c5a5a5a,even odd rule,draw opacity=0.993,line join=bevel,line
	width=0.000cm,miter limit=3.80] (10.4266,14.680) node[above right]
	(text5674){$x$};

	\path[fill=c5a5a5a,even odd rule,draw opacity=0.993,line join=bevel,line
	width=0.000cm,miter limit=3.80] (10.3535,16.8859) node[above right]
	(text6718){$z_x$};

	\path (10.0223,15.3885)arc(270.200:334.400:0.563040 and 0.532) --
	(10.5301,15.1584) -- cycle;

	\draw (10.5230,
	16) arc [start angle = 98, end angle=149.5, x radius =1, y radius =1];
	
	\path[draw=c595959,even odd rule,line join=bevel,line width=0.000cm,miter
	limit=3.80,dash phase=0.101cm] (10.1659,15.4452) node[above right, scale =0.8]
	(text10402){$\beta$};

	\path[draw=black,fill=black,line join=bevel,line width=0.102cm] (10.5230,
	15.1654) ellipse (0.1021cm and 0.1015cm);
	
	\end{tikzpicture}
	\caption{Geometry close to $\partial \Omega$. Here $\partial \widetilde{\Omega}$ is the surface $\partial \Omega$ reflected with respect to $x$. The grey shaded area visualizes the integration domain of the term $(\text{II})$. The angle $\beta$ is close to $\pi/2$ by the $C^{1,\alpha}$-boundary and the choice of $a$.}
\label{fig:geometry_lemma_A_10}
\end{figure}
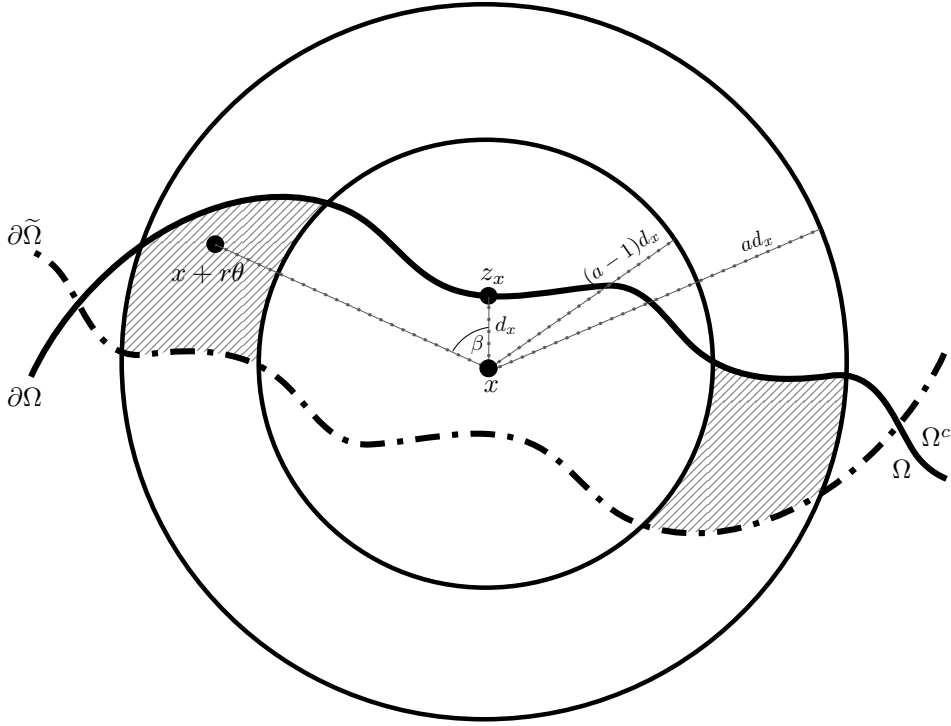
	
	Finally, we insert this into the term $(\text{II})$ to find:
	\begin{align*}
		(\text{II})\le  \mu(\mathbb S^{d-1}) \Big(\frac{ 2 }{a-1 }\Big)^{2s} \int\limits_{(a-1) d_x}^{a d_x} \frac{1}{r^{1+2s}}\d r\le \mu(\mathbb S^{d-1}) \frac{\big((a-1)^{-2s}-a^{-2s}\big)}{2s} \frac{2^2}{(a-1)^{2s}} d_x^{-2s}.
	\end{align*}
	This puts us into the position to absorb the term $(\text{II})$ to the left-hand side of \eqref{eq:calculation_lemma_distance_estimate}: 
	\begin{align*}
		\lambda \,\frac{\big( (a-1)^{-2s}-a^{-2s} \big)}{4s}  d_x^{-2s}\le (\text{I})\le  \int\limits_{\abs{r}<ad_x} \,\int\limits_{\mathbb S^{d-1}} \frac{\1_{\Omega^c}(x+r\theta)}{\abs{r}^{1+2s}}\mu(\d\theta) \, \d r.
	\end{align*}	
\end{proof}

\smallskip 

Recall that $\Gamma_s$ is the carré du champ for the operator $A_s$, see \eqref{eq:carre_du_champ}. 

\begin{corollary}\label{cor:comparison_forms}
	Let $\Omega\subset \R^d$ be a bounded $C^{1,\alpha}$-domain with $\alpha\in (0,1)$ and $s_\star\in (0,1)$. There exists a constant $C=C(d,\Omega, \alpha, s_\star,\lambda, \Lambda)>0$ such that for any $s\in (s_\star,1)$:
	\begin{align*}
		(1-s)\int\limits_{\Omega\,\Omega} \big(d_x\wedge d_y \big)^s \frac{\abs{u(x)-u(y)}^2}{\abs{x-y}^{d+2s}}\d y \d x\le C \Big(\int\limits_{\Omega} d_x^{s} \Gamma_s(u)(x)\d x + (1-s)\norm{u}_{L^2(\Omega)}^2+ \norm{u}_{L^2(\Omega^c; \tau_s)}^2\Big) 
	\end{align*}
	for any $u:\R^d\to \R$ measurable. 
\end{corollary}

\begin{proof}
	The proof follows from \autoref{lem:comparison_forms} with $\beta =s$, \autoref{prop:distance_estimate} and triangle inequality adding $\pm u(x+r\theta)$ in the term $\norm{u}_{L^2(\Omega;d_x^{-s})}$ with a calculation similar to the proof of \autoref{lem:normal_derivative_distance}. 
\end{proof}

\subsection{A robust weighted Sobolev inequality}\label{sec:weighted_sobolev}

\begin{lemma}\label{lem:weighted_sobolev}
	Let $s_\star\in (0,1)$ and $\Omega\subset \R^d$ be a bounded Lipschitz domain. There exists a constant $C=C(d,\Omega, s_\star)\ge 1$ such that for any $s\in(s_\star, 1)$
	\begin{equation}
		\int\limits_{\Omega}\int\limits_{\Omega} \frac{\abs{u(x)-u(y)}^2}{\abs{x-y}^{d+s}}\d y \d x \le C\,(1-s) \int\limits_{\Omega} \int\limits_{\Omega}(d_x\wedge d_y)^s \frac{\abs{u(x)-u(y)}^2}{\abs{x-y}^{d+2s}}\d y \d x 
	\end{equation}
	for any $u:\Omega \to \R$ such that the right-hand side is finite.
\end{lemma}
\begin{proof}
	We divide the proof into three steps. First we prove a localized version of the statement and then pick a Whitney decomposition to prove the statement on $\Omega$. 
	
	\smallskip 
	
	\textit{Step 1:} We show that there exists a constant $c_1=c_1(d,s_\star)\ge 1$ such that for any ball $B_r$ with radius $r>0$ and any function $v \in H^s(\mathbb{R}^d)$ we find
	\begin{equation*}
		\int\limits_{B_r}\int\limits_{B_r} \frac{\abs{v(x)-v(y)}^2}{\abs{x-y}^{d+s}}\d y \d x \le c_1 (1-s) r^s\, \int\limits_{B_r}\int\limits_{B_r} \frac{\abs{v(x)-v(y)}^2}{\abs{x-y}^{d+2s}}\d y \d x.
	\end{equation*}
	Without loss of generality we assume the ball $B_r$ to be centered at the origin by translation invariance. We prove the statement for $r=1$ and then conclude by scaling. Let $E:H^{s}(B_1)\to H^{s}(\R^d)$ be a Sobolev extension operator with the operator norm being bounded by $c_2=c_2(d, s_\star)\ge 1$, see the proofs of \cite[Theorem 5.4, Lemma 5.1, 5.2, 5.3]{DPV12}. The Sobolev embedding $H^s(\R^d)\hookrightarrow H^{s/2}(\R^d)$ yields a constant $c_3=c_3(d, s_\star)\ge 1$ such that 
	\begin{align*}
		\int\limits_{B_1}\int\limits_{B_1} \frac{\abs{v(x)-v(y)}^2}{\abs{x-y}^{d+s}}\d y \d x \le (1-s/2)^{-1} [Ev]_{H^{s/2}(\R^d)}^2 \le c_3 \norm{Ev}_{H^{s}(\R^d)}^2\le c_2 c_3 \norm{v}_{H^{s}(B_1)}^2.
	\end{align*}
	Since the seminorms in the previous inequality are invariant under additive constants, we can apply it to $w:=v-\fint_{B_1}v$ and use a robust Poincaré inequality $\norm{w}_{L^2(B_1)}\le c_4 [v]_{H^s(B_1)}^2$, see \cite{BBM02} and \cite{Pon04}, with a constant $c_4=c_4(d,s_\star)\ge 1$ to find
	\begin{equation}\label{eq:final_step_cubes_size_1}
		\int\limits_{B_1}\int\limits_{B_1} \frac{\abs{v(x)-v(y)}^2}{\abs{x-y}^{d+s}}\d y \d x \le c_2c_3 (1+c_4) (1-s)\int\limits_{B_1}\int\limits_{B_1} \frac{\abs{v(x)-v(y)}^2}{\abs{x-y}^{d+2s}}\d y \d x.
	\end{equation}
	Now, scaling this inequality to balls with radius $r$ yields the claim. 
	
	\smallskip 
	
	\textit{Step 2:} It is proven in \cite{Dyd06} that there exists a constant $c_5=c_5(d,\Omega, s_\star)\ge 1$ such that 
	\begin{equation*}
		\int\limits_{\Omega}\int\limits_{\Omega} \frac{\abs{u(x)-u(y)}}{\abs{x-y}^{d+s}}\d y \d x \le c_5 	\int\limits_{\Omega}\int\limits_{B_{d_x/2^6}(x)} \frac{\abs{u(x)-u(y)}^2}{\abs{x-y}^{d+s}}\d y \d x.
	\end{equation*}
See also \autoref{rem:dyda_robust} for the $s$-dependence of the constant $c_5$.
	
	\smallskip
	
	\textit{Step 3:} Let $\cW$ be a Whitney ball covering of $\Omega$ as in step 3 in the proof of \autoref{lem:comparison_forms}. In particular, recall that for any ball $B\in \cW$ we denote by $r_B$ its radius, by $B^\star$ a slightly bigger ball with radius $(7/4)r_B$ and by $x_B$ its center. These slightly bigger balls satisfy a bounded overlap property with a constant $N\in \N$, see \eqref{eq:whitney_balls}.
	
	\smallskip
	
	Now, we use step 2 and decompose the integral over $\Omega$ into these Whitney balls:
	\begin{align*}
		\int\limits_{\Omega}\int\limits_{\Omega} \frac{\abs{u(x)-u(y)}}{\abs{x-y}^{d+s}}\d y \d x \le \sum\limits_{B\in \cW} 	\, \int\limits_{B}\int\limits_{B_{d_x/8}(x)} \frac{\abs{u(x)-u(y)}}{\abs{x-y}^{d+s}}\d y \d x.
	\end{align*}
	In the beginning of step 3 in the proof of \autoref{lem:comparison_forms} we proved that for any $B\in \cW$, $x\in B$ and $y\in B_{d_x/8}(x)$ both $x$ and, more importantly, $y$ are in the slightly bigger ball $B^\star$. We apply this and step 1 to find:
	\begin{equation*}
		\int\limits_{B}\int\limits_{B_{d_x/8}(x)} \frac{\abs{u(x)-u(y)}}{\abs{x-y}^{d+s}}\d y \d x\le \int\limits_{B^\star}\int\limits_{B^\star} \frac{\abs{u(x)-u(y)}}{\abs{x-y}^{d+s}}\d y \d x\le c_1 (1-s) (7/4)^s\,r_B^s \int\limits_{B^\star} \int\limits_{B^\star} \frac{\abs{u(x)-u(y)}^2}{\abs{x-y}^{d+2s}}\d y \d x.
	\end{equation*}
	By the properties of the Whitney covering, see \eqref{eq:whitney_balls}, we find that $r_B^s\le 4^{s}(d_x\wedge d_y)^s$ for $x,y\in B^\star$. Thus, summing over all balls yields
	\begin{align*}
		\int\limits_{\Omega}\int\limits_{\Omega} \frac{\abs{u(x)-u(y)}}{\abs{x-y}^{d+s}}\d y \d x &\le \sum\limits_{B\in \cW} 	2\,c_1 (1-s) r_B^s \int\limits_{B^\star} \int\limits_{B^\star} \frac{\abs{u(x)-u(y)}^2}{\abs{x-y}^{d+2s}}\d y \d x\\
		&\le 8\,c_1 N (1-s)\int\limits_{\Omega} \int\limits_{\Omega} (d_x\wedge d_y)^s \frac{\abs{u(x)-u(y)}^2}{\abs{x-y}^{d+2s}}\d y \d x.
	\end{align*}
	Here we used the finite overlap property of the family of slightly bigger balls.
\end{proof}

\section{Proof of \autoref{th:main_result}}\label{sec:proof}

We structure the proof as follows: First we show that solutions to \eqref{eq:main_equation} for regular $f,g$ satisfy the bound \eqref{eq:main_estimate_theorem}. This will be done in the following two lemmata. Then we conclude the result by approximating rough $f,g$ as in \autoref{th:main_result} by a sequence of regular functions $f_n, g_n$. 

\begin{lemma}\label{lem:L2_bound}
	Let $\Omega\subset \R^d$ be a bounded $C^{1,\alpha}$-domain, $\alpha \in (0,1)$, $s_\star\in (0,1)$ and $f\in C_c^\infty(\Omega)$, $g\in C_c^\infty(\Omega^c)$. For any $s\in (s_\star,1)$ and any classical solution $u\in C^s(\R^d)\cap C_{\loc}^{3s}(\Omega)$ to \eqref{eq:main_equation} we have
	\begin{equation}\label{eq:L2_bound}
		\norm{u}_{L^2(\Omega)}\le C\, \Big(\norm{f}_{L^2(\Omega; d_x^{2s})}  + \norm{g}_{L^2(\Omega^c; \tau_s)} \Big)
	\end{equation}
	for some constant $C=C(d,\Omega, \alpha, s_\star, \lambda, \Lambda)>0$.
\end{lemma}
\begin{proof}
	Let $\phi$ be the classical solution to the Dirichlet problem $A_s \phi =1$ in $\Omega$, $\phi=0$ on $\Omega^c$ from \autoref{lem:phi_test_function} and the constant $c_1=c_1(d,\Omega, \alpha, s_\star, \lambda,\Lambda)\ge 1$ such that $\phi(x)\le c_1 d_x^s$. In the following calculation, we use a nonlocal integration by parts formula, after which we need to estimate the term $N_s(\phi)(x)$, $x\in \Omega^c$. By \autoref{lem:normal_derivative_distance} we find a constant $c_2=c_2(d,\Omega, \alpha, s_\star, \lambda,\Lambda)\ge 1$ such that $-N_s(\phi)(x)\le -c_1N_s(d_x^s)(x)\le c_2\tau_s(x)$. An application of the nonlocal Gauss-Green formula yields
	\begin{align*}
		\int\limits_{\Omega} \abs{u(x)}^2 \d x &= \int\limits_{\Omega} A_s\phi(x) \abs{u(x)}^2\d x= \int\limits_{\Omega} \phi(x) A_s(u^2)(x)\d x - \int\limits_{\Omega^c} N_s(\phi)(x)\abs{g(x)}^2 \d x\\
		&= 2 \int\limits_{\Omega} \phi (x) u(x) f(x)\d x - \int\limits_{\Omega} \phi (x)\Gamma_s(u)(x)\d x- \int\limits_{\Omega^c} N_s(\phi)(x)\abs{g(x)}^2 \d x.
	\end{align*}
	In the last inequality we used
	\begin{equation*}
			\Gamma_s(u) = -A_s(u^2) + 2uf \quad\text{in $\Omega$}.
	\end{equation*}
	Since $\phi$ and $\Gamma_s(u)$ are nonnegative, the second term is nonpositive. In addition, we use $a\, b \le 1/8 a^2 + 2^{5} b^2$, for any real $a,b$, in the first term and $-N_s(\phi)\le c_2\tau_s$ in the last term to find
	\begin{align*}
		\int\limits_{\Omega} \abs{u(x)}^2 \d x\le \frac{1}{4}\int\limits_{\Omega} \abs{u(x)}^2 \d x + 2^6 \int\limits_{\Omega} \abs{\phi(x)\, f(x)}^2 \d x+ c_2 \int\limits_{\Omega^c} \abs{g(x)}^2 \tau_s(\d x).
	\end{align*}
	Absorbing the $L^2$-norm of $u$ and using the upper distance bound on $\phi$, we find the desired estimate:
	\begin{equation*}
		\int\limits_{\Omega} \abs{u(x)}^2\d x \le c_1^2 \frac{2^{8}}{3} \int\limits_{\Omega} d_x^{2s}\abs{f(x)}^2 \d x+ c_2\frac{4}{3} \int\limits_{\Omega^c} \abs{g(x)}^2 \tau_s(\d x).
	\end{equation*}
\end{proof}

\begin{lemma}\label{lem:main_result}
	Let $\Omega\subset \R^d$ be a bounded $C^{1,\alpha}$-domain, $\alpha \in (0,1)$, $s_\star\in (0,1)$ and $f\in C_c^\infty(\Omega)$, $g\in C_c^\infty(\Omega^c)$. For any $s\in (s_\star,1)$ and any classical solution $u\in C^s(\R^d)\cap C_{\loc}^{3s}(\Omega)$ to \eqref{eq:main_equation} the estimate \eqref{eq:main_estimate_theorem} is satisfied.
\end{lemma}
\begin{proof}
By \autoref{lem:L2_bound}, $u$ satisfies the correct $L^2$-bound \eqref{eq:L2_bound}. Thus, it suffices to estimate the seminorm on the left-hand side of \eqref{eq:main_estimate_theorem}. Now let $\phi \in C^s(\mathbb R^d) \cap C^{3s}_{\loc}(\Omega)$ be the classical solution to $A_s\phi =1$ in $\Omega$ and $\phi =0$ on $\Omega^c$ from \autoref{lem:phi_test_function}. Further, let $c_1=c_1(d,\Omega, \alpha, s_\star, \lambda, \Lambda)\ge 1$ such that $c_1^{-1}d_x^s\le \phi(x)\le c_1 d_x^s$. By \autoref{lem:weighted_sobolev} and \autoref{cor:comparison_forms} there exists a constant $c_2=c_2(d,\Omega,\alpha, s_\star, \lambda, \Lambda)\ge 1$ such that
\begin{equation}\label{eq:seminorm_bound_help1}
	[u]_{H^{s/2}(\Omega)}^2  \le c_2  \Big(\int\limits_{\Omega} d_x^{s} \Gamma_s(u)(x)\d x + (1-s)\norm{u}_{L^2(\Omega)}+ \norm{g}_{L^2(\Omega^c; \tau_s)}^2\Big).
\end{equation}
Again, after estimating $d_x^s$ by $\phi$, we use
\begin{equation*}
	\Gamma_s(u) = -A_s(u^2) + 2uf \quad\text{in $\Omega$}.
\end{equation*}
Thus, the first term on the right-hand side of \eqref{eq:seminorm_bound_help1} is bounded by
\begin{equation*}
	-c_1c_2\int_{\Omega} \phi(x) A_s(u^2)(x) dx + 2c_1^2c_2 \int_{\Omega} d_x^s \abs{u(x)f(x)} dx =: (\text{I}) + (\text{II}).
\end{equation*}
Since $u$ and $\phi$ are classical solution with smooth data, we can use the nonlocal Gauss-Green formula to find:
\begin{align*}
	(\text{I}) &= -c_1 c_2\int_{\Omega} \big[A_s\phi\big] (x)\, \abs{u(x)}^2 \d x - c_1c_2\int_{\Omega^c} \abs{g(y)}^2  N_s (\phi)(y) \d y\\
	&= -c_1c_2\int_{\Omega} \abs{u(x)}^2 \d x - c_1c_2\int_{\Omega^c} \abs{g(y)}^2 N_s (\phi)(y) \d y.
\end{align*}
The first term is nonpositive and for the second term we use \autoref{lem:normal_derivative_distance} to estimate $-N_s(\phi)\le c_1c_3 \tau_s$ in $\overline{\Omega}^c$, for a constant $c_3=c_3(\Omega,s_\star, \Lambda )\ge 1$. Hence, 
\begin{equation*}
	(\text{I}) \le c_1^2 c_2 c_3 \norm{g}_{L^2(\Omega;\tau_s)}^2.
\end{equation*}

\smallskip

Now we estimate $(\text{II})$. An application of Hölder's inequality together with \autoref{lem:L2_bound} yields
\begin{equation*}
	(\text{II})\le 2c_1^2 c_2\norm{u}_{L^2(\Omega)} \norm{f}_{L^2(\Omega;d_x^{2s})} \le 2 c_1^2 c_2 c_4 (\norm{g}_{L^2(\Omega^c;\tau_s)} +  \norm{f}_{L^2(\Omega;d_x^{2s})} )\norm{f}_{L^2(\Omega;d_x^{2s})}
\end{equation*} 
with a constant $c_4 =c_4(d,\Omega, \alpha, s_\star, \lambda, \Lambda)>0$.
\end{proof}

\begin{proof}[Proof of \autoref{th:main_result}]
	Let $f\in L^2(\Omega;d_x^{2s})$ and $g \in L^2(\Omega^c;\tau_s)$. By the density of $C_c^\infty$-functions in $L^2$-spaces with Radon measures on $\mathbb R^d$, there exist sequences $\{f_m\}$ and $\{g_m\}$ where $f_m \in C^\infty(\overline{\Omega})$ and $g_m \in C_c^\infty(\Omega^c)$ such that
	\begin{equation*}
		\lim\limits_{m \to \infty}\norm{f_m -f}_{L^2(\Omega;d_x^{2s})} = 0 \text{  and  }\lim\limits_{m\to \infty}\norm{g_m-g}_{L^2(\Omega^c;\tau_s)}=0.
	\end{equation*}
	Let $u_m \in C^s(\mathbb R^d)\cap C_\loc^{3s}(\Omega)$ be the classical solution to \eqref{eq:main_equation} with the data $f_m$ and $g_m$. The existence of such $u_m$ follows as in \autoref{lem:phi_test_function}. By \autoref{lem:L2_bound} and \autoref{lem:main_result}, we find a constant $c_1=c_1(d,\Omega, \alpha, s_\star, \lambda, \Lambda)>0$ such that
	\begin{equation*}
		\norm{u_m}_{H^{s/2}(\Omega)} \le c_1\Big( \norm{f_m}_{L^2(\Omega;d_x^{2s})}+\norm{g_m}_{L^2(\Omega^c;\tau_s)} \Big).
	\end{equation*}
	By the linearity of \eqref{eq:main_equation}, the sequence $\{u_m\}$ is a Cauchy sequence in $H^{s/2}(\Omega)$ and there exists a limit $\tilde{u} \in H^{s/2}(\Omega)$. For every $\phi \in C_c^{\infty}(\Omega)$ it holds that $\abs{\phi(x)} \le \norm{\phi}_{C^s(\overline{\Omega})} d_x^s$. Thus, an application of Hölder's inequality yields
	\begin{equation*}
		(f_m-f,\phi)_{L^2(\Omega)} \le \norm{\phi}_{C^s(\overline{\Omega})} \abs{\Omega}^{1/2} \norm{f_m-f}_{L^2(\Omega;d_x^{2s})} \longrightarrow 0 \quad\text{for $m\to\infty$}.
	\end{equation*} 
	Since $A_s(\phi) \in L^{\infty}(\Omega)$ and $\tilde{u}$ is the $H^{s/2}(\Omega)$-limit of $u_m$, we find
	\begin{equation*}
		\lim\limits_{m \to \infty}(u_m-\tilde{u},A_s(\phi))_{L^2(\Omega)} =0.
	\end{equation*}  
	As in \autoref{lem:normal_derivative_distance}, the operator $A_s$ acting on the test function $\phi$ on $\Omega^c$ is bounded by a multiple of $\tau_s$, \ie $N_s(\phi)(x) \le c_2  \tau_s(x)$ for $x\in \Omega^c$ with $c_2 = c_2(d, \Omega, \alpha, s_\star,\lambda, \Lambda)$. Since $\tau_s$ is a finite measure on $\Omega^c$, we conclude 
	\begin{equation*}
		\lim\limits_{m \to \infty}(g_m-g,A_s(\phi))_{L^2(\Omega^c)} =0.
	\end{equation*}  
	We define $u:= \tilde{u}$ in $\Omega$ and $u:= g$ on $\Omega^c$. Hence, $u$ is a distributional solution to \eqref{eq:main_equation}. Moreover, $\eqref{eq:main_estimate_theorem}$ holds for $u$, as $u_m$ satisfies \eqref{eq:main_estimate_theorem} uniformly. 
	
	\smallskip 
	
	Lastly, we prove the uniqueness of distributional solutions in $H^{s\slash 2}(\Omega)$. Suppose there is another distributional solution $v\in H^{s/2}(\Omega)$. Then $u-v$ is a distributional solution to the problem \eqref{eq:main_equation} with $f=0$ and $g=0$. Since $u-v\in H^{s/2}(\Omega)$, the fractional Hardy inequality, see \cite[Theorem 1.1 (17)]{Dyd04} or \cite[Theorem 2.3]{ChSo03}, yields a constant $c_4$ such that for $\eps>0$
	\begin{equation}\label{eq:condition_maximum}
		\eps^{-s}\int\limits_{\Omega\setminus \Omega_\eps} \abs{u(x)-v(x)}\d x \le c_3 \Big(\int\limits_{\Omega\setminus \Omega_\eps} \frac{\abs{u(x)-v(x)}^2}{d_x^{s}}\d x \Big)^{1/2}\le c_4 \norm{u-v}_{H^{s/2}(\Omega\setminus \Omega_\eps)} \to 0
	\end{equation}
	 as $\eps \to 0+$. Recall that $\Omega_\eps = \{ x\in \Omega\,|\, d_x>\eps \}$. The maximum principle for distributional solutions, see \cite[Theorem 1.1]{GrHe22a}, yields $u=v$. Note that this is applicable as described in \cite[Remark 1.4]{GrHe22a} together with \cite[Proposition 2.6.9]{Ros23}.
\end{proof}

\section{Nonlocal to Local}\label{sec:nonlocaltolocal}

	As an application of the robust regularity result \autoref{th:main_result} and the convergence of trace spaces from \cite{GrHe22b}, we can approximate distributional solutions of the Dirichlet problem for the Laplacian with distributional solutions to \eqref{eq:main_equation} and receive $H^{1/2}(\Omega)$-regularity for solutions to \eqref{eq:dirichlet_laplace} with $L^2$-boundary data. One difficulty in the setup of the local Dirichlet problem with data $g \in L^2(\partial\Omega)$ is the question of how to prescribe the boundary data. The obstruction is that a distributional solution does not have sufficient regularity a priori such that it allows for a description of its boundary values. There are multiple ways of solving this problem. One of them is to describe boundary values by nontangential convergence, cf.\ \cite{HuntWheeden, DahlbergEstimatesHarmonicMeasure, JK81}. Another way is to include the boundary data into the equation by enlarging the test function space.
\begin{definition}
	Let $\Omega \subset \mathbb R^d$ be a bounded Lipschitz domain, $f\in L^1(\Omega; d_x)$ and $g\in L^1(\partial\Omega)$. A function $u\in L^1(\Omega)$ is a very weak solution to
	\begin{align}\label{eq:dirichlet_laplace_with_f}
		\begin{split}
			-\Delta u &= f \quad \text{ in }\Omega,\\
			u&=g \quad \text{ on }\partial \Omega
		\end{split}
	\end{align}
	 if for all $\phi \in C^1(\overline{\Omega})$ such that $\phi = 0$ on $\partial \Omega$, and $\Delta \phi \in L^\infty(\Omega)$ the following holds
	\begin{equation}\label{eq:new_local_solution}
		\int\limits_{\Omega} u (-\Delta) \phi \d x = \int\limits_{\Omega} f \phi \d x - \int\limits_{\partial\Omega} g \partial_n \phi \d \sigma.
	\end{equation}
	We denote the space of test functions by $C_0^{1, \Delta}(\overline{\Omega})$.
\end{definition}
This definition of a solution was used in \cite[Definition 2.3]{veron} and is motivated by the Gauss-Green theorem. Very weak solutions are well-posed, see \cite[Theorem 2.4]{veron}. In \cite[Proposition 20.2]{Po16}, it was shown that, for $g=0$, solutions in the sense of \eqref{eq:new_local_solution} satisfy
\begin{equation}\label{eq:condition_max_principle_0_data}
 \lim\limits_{\varepsilon \to 0+}	\eps^{-1}\int\limits_{\Omega\setminus \Omega_\eps} \abs{u(x)}\d x = 0.
\end{equation}
 Additionally, every function $u \in L^1(\Omega) $ satisfying $(u,-\Delta \phi)_{L^2(\Omega)} =(f, \phi)_{L^2(\Omega)}$ for all $\phi \in C_c^\infty(\Omega)$ and \eqref{eq:condition_max_principle_0_data} is a solution in the sense of \eqref{eq:new_local_solution} for $g=0$.

In analogy to \eqref{eq:new_local_solution} we have the following:

\begin{proposition}\label{prop:new_solution}
	Let $\Omega$ be a bounded $C^{1,\alpha}$-domain with $\alpha \in (0,1)$. Let $s \in (0,1)$, $f \in L^2(\Omega;d_x^{2s})$ and $g \in L^2(\Omega^c;\tau_s)$. Then the solution $u$ from \autoref{th:main_result} satisfies
	\begin{align}\label{def:solution}
		\begin{split}
			\int\limits_{\Omega} u(x)A_s(\phi)(x) \d x &= \int\limits_{\Omega}f(x)\phi(x) \d x - \int\limits_{\Omega^c} g(x) N_s(\phi)(x) \d x \quad\text{for all $\phi \in C_0^{s,A_s}(\overline{\Omega})$}\\
			u&=g \quad\text{on $\Omega^c$.}
		\end{split}
	\end{align}
	With $C_0^{s,A_s}(\overline{\Omega})$ we denote the set of functions $\phi : \mathbb{R}^d \to \mathbb{R}$ such that $\phi \in C^s(\overline{\Omega})$, $\phi = 0$ on $\Omega^c$ and $A_s(\phi) \in L^\infty(\Omega)$. 
\end{proposition}
\begin{proof}
	As in the proof of \autoref{th:main_result}, we choose approximating sequences $f_m\in C_c^\infty(\Omega)$ and $g_m\in C_c^\infty(\Omega^c)$ for $f$ and $g$, and $u_m$ be the corresponding solutions to \eqref{eq:main_equation}. Since $u_m \in C^s(\mathbb R^d)\cap C^{3s}_{\loc}(\Omega)$, we can use Gauss-Green to verify \eqref{def:solution} for $u_m$. 
	
	Thus, it suffices to show that every integral converges for $m\to\infty$. Let $\phi\in C_0^{s, A_s}(\Omega)$. The convergence
	\begin{equation*}
		\int\limits_\Omega u_m(x) A_s(\phi)(x) \d x \to \int\limits_\Omega u(x) A_s(\phi)(x) \d x
	\end{equation*}
	follows, as $u_m \to u$ in $L^2(\Omega)$ and $A_s(\phi) \in L^\infty(\Omega)$.
	
	For the inhomogeneity we use that $\phi \in C^s(\overline{\Omega})$ and $0$ outside of $\Omega$ to get
	\begin{equation*}
		\int\limits_{\Omega} |f(x) - f_m(x)| |\phi(x)| \d x \leq \norm{\phi}_{C^s(\overline{\Omega})} \int\limits_{\Omega} |f(x) - f_m(x)| d_x^s \d x \leq \norm{\phi}_{C^s(\overline{\Omega})} |\Omega|^{1\slash 2}\norm{f - f_m}_{L^2(\Omega; d_x^{2s})}.
	\end{equation*}
	
	The exterior integral is similar. Note that $\phi$ is a classical solution with $A_s\phi \in L^\infty(\Omega)$ and $\phi = 0$ on $\Omega^c$. Hence, there is a constant $c_1 > 0$ such that $|\phi(x)| \leq c_1 d_x^s$ for $x\in \Omega$. In particular, by \autoref{lem:normal_derivative_distance}, we have $|N_s(\phi)| \leq c_2 \tau_s$ on $\Omega^c$, for some $c_2 > 0$. 
\end{proof}

\begin{theorem}\label{th:nonlocalTolocal}
	Let $\Omega$ be a bounded smooth domain, $g \in L^2(\partial\Omega)$ and $f \in L^2(\Omega;d_x^2)$. There exist a sequence $s_n \in (0,1)$, converging to $1$, $g_n \in L^2(\Omega^c;\tau_{s_n})$, $f_n \in L^2(\Omega; d_x^{2s_n})$ and distributional solutions $u_n \in H^{s/2}(\Omega)$ to
	\begin{align*}
		(-\Delta)^{s_n} u_n &= f_n \quad\text{in }\Omega, \\
		u_n &= g_n \quad\text{on }\Omega^c,
	\end{align*}  
	satisfying \eqref{eq:main_estimate_theorem}. The sequence $u_n$ converges almost everywhere and in $L^2(\Omega)$ to the unique very weak solution $u\in L^1(\Omega)$ of \eqref{eq:dirichlet_laplace_with_f} with data $f$ and $g$ in the sense of \eqref{eq:new_local_solution}. Moreover, $u\in H^{1\slash 2}(\Omega)$ and there exists a constant $C=C(d,\Omega)$ such that
	\begin{equation}\label{eq:local_main_estimate_theorem}
		\norm{u}_{H^{1/2}(\Omega)} \le   C \Big( \norm{f}_{L^2(\Omega;d_x^{2})}+ \norm{g}_{L^2(\partial\Omega)} \Big).
	\end{equation}
\end{theorem}
\begin{proof}
	We just sketch the proof.
	
	\textit{Step 1:} We equip the space of test-functions $C^{1,\Delta}_0(\overline{\Omega})$ with the norm
	\begin{equation*}
		\norm{\Delta \phi}_{L^2(\Omega)} + \norm{\phi}_{L^2(\Omega;d_x^2)} + \norm{\partial_n \phi}_{L^2(\partial\Omega)}.
	\end{equation*} 
	Then $C^2_0(\overline{\Omega})= \{v \in C^2(\overline{\Omega}) \mid v = 0 \text{ on } \partial\Omega \}$ is dense in $C^{1,\Delta}_0(\overline{\Omega})$. \smallskip
	
	\textit{Step 2:} We reduce the claim to $g\in H^{1/2}(\partial\Omega)$ by approximating $g\in L^2(\partial\Omega)$ with a sequence of elements $g_n \in H^{1/2}(\partial\Omega)$. If the solutions $u_n$ satisfy the bound
	\begin{equation*}
		\norm{u_n}_{H^{1/2}(\Omega)} \le c (\norm{f}_{L^2(\Omega;d_x^2)} + \norm{g_n}_{L^2(\partial\Omega)}),
	\end{equation*} 
	then the limit $u$ satisfies the same bound with $g_n$ replaced by $g$. \smallskip
	
	\textit{Step 3:} For $g \in H^{1/2}(\partial\Omega)$ there exists an extension $Eg \in H^1(\mathbb{R}^d)$. We set $f_s := f d_x^{1-s} \in L^2(\Omega;d_x^{2s})$ and $g_s := Eg\vert_{\Omega^c} \in L^2(\Omega^c;\tau_s)$. By \autoref{th:main_result} and \autoref{prop:new_solution}, there exist distributional solutions $u_s \in H^{s/2}(\Omega)$ to \eqref{eq:main_equation} with data $f_s$ and $g_s$ satisfying \eqref{def:solution} and
	\begin{equation*}
		\norm{u_s}_{H^{s/2}(\Omega)} \le c \Big(\norm{f_s}_{L^2(\Omega;d_x^{2s})}+\norm{g_s}_{L^2(\Omega^c;\tau_s)}\Big).
	\end{equation*} 
	Since $H^{s_\star/2}(\Omega)$ is compactly embedded into $L^2(\Omega)$, we find $u \in L^2(\Omega)$ and a sequence $s_n \in (0,1)$, converging to $1$, such that $u_n := u_{s_n} \to  u$ as $n \to \infty$ almost everywhere and in $L^2(\Omega)$. Fatou's Lemma implies
	\begin{align*}
		\norm{u}_{H^{1/2}(\Omega)} &\le  \lim\limits_{n\to \infty}\norm{u_n}_{H^{s_n/2}(\Omega)} \le c  \lim\limits_{n\to \infty} \Big( \norm{f_{s_n}}_{L^2(\Omega;d_x^{2s_n})} +  \norm{g_{s_n}}_{L^2(\Omega^c;\tau_{s_n})} \Big)\\ 
		&= c \Big( \norm{f}_{L^2(\Omega;d_x^{2})}+ \norm{g}_{L^2(\partial\Omega)} \Big).
	\end{align*}
	Here we used that $\norm{f_s}_{L^2(\Omega; d_x^{2s})} = \norm{f}_{L^2(\Omega, d_x^2)}$ and \cite[Theorem 1.4]{GrHe22b}, implying
	\begin{equation*}
		\lim\limits_{s \to 1-} \norm{Eg}_{L^2(\Omega^c;\tau_s)}  =\norm{g}_{L^2(\partial\Omega)}.
	\end{equation*}
	It is left to show that $u$ is a solution to \eqref{eq:dirichlet_laplace_with_f} in the sense of \eqref{eq:new_local_solution}. For a function $\phi \in C^2_0(\overline{\Omega})$ it follows from \cite[Lemma 5.75]{Fog20} that
	\begin{equation*}
		\lim\limits_{s\to 1-} \int\limits_{\Omega^c} Eg N_s(\phi) \d x = \int\limits_{\partial\Omega} g \partial_n \phi \d \sigma
	\end{equation*} 
	This and step 1 yield that $u $ is a solution to \eqref{eq:dirichlet_laplace_with_f} and satisfies \eqref{eq:local_main_estimate_theorem}.
\end{proof}

\appendix 
\section{Meta Analysis of the Dependencies of Constants}\label{sec:appendix}

\begin{remark}\label{rem:ros_oton_constant}
	Fix a $C^{1, \alpha}$-domain $\Omega$ and let $u$ be a classical solution to the exterior data problem
	\begin{align*}
			A_s u &= 1 \quad \text{ in $\Omega$}, \\
			\hspace{12pt}	u &= 0 \quad\text{ on $\Omega^c$}.
	\end{align*}
	It was proved in \cite{Ros23} that there exist constants $c, C>0$ such that
	\begin{equation}\label{eq:appendix}
		c d_x^s \leq u(x) \leq C d_x^s \quad\text{for every $x\in\Omega$}.
	\end{equation}
	We will now argue that these constants $c, C>0$ can be chosen depending on $s$ only by a lower bound $0<s_\star \leq s < 1$. 
	\smallskip 
	
	In the following, every Proposition and Lemma refers to \cite{Ros23}. In particular, most of the constants will be named after their corresponding Proposition\slash Lemma in \cite{Ros23}. Also it is important to note that our definition of $A_s$ has a factor $1-s$, whereas in \cite[(2.1.15)]{Ros23} the factor $c_s$ is neither specified nor used.
	
	Except for the proof of the Hopf-Lemma, \ie the lower bound in \eqref{eq:appendix}, the results which we cite in this section are contained in the article \cite{RoSe17}. But for a consistent numbering of the constants we use the notation and numbering of theorems from the book \cite{Ros23}.
	
	The Proposition 2.6.6 (Hopf Lemma) implies the lower bound and within the proof of Proposition 2.6.4 the author proves the upper bound. Let us only discuss the lower bound done in Proposition 2.6.6., as the upper one has similar steps. The constant $c$ is given by $c_\ast \cdot c_{B.2.7} \cdot \kappa$, where
	\begin{equation*}
		c_\ast = \min \{ u(x) \mid x\in \Omega_\delta\},
	\end{equation*}
	and $c_{B.2.7}$, $0<\delta < 1$ are taken from Corollary B.2.7, and $\kappa > 0$ is a constant coming from geometric arguments and can therefore be chosen independent of $s$. Recall that $\Omega_\delta= \{ y\in \Omega \mid d_y>\delta \}$.
	
	By a comparison principle argument one can show that $c_\ast$ depends only on $\delta$ and $s_\star$ by using, for example, the following barrier function at $x_0\in \Omega_\delta$
	\begin{equation*}
		v(x) = K^{-1}\max\left\{\left(\frac{\delta}{2} \right)^2 - |x-x_0|^2, -2\delta \right\},
	\end{equation*}
	where $K$ is a  constant depending only on $\delta$ and $s_\star$ such that $A_s v \leq 1$ in $B_{\delta\slash 2}(x_0)$. This implies $u \ge v$ in $B_{\delta\slash 2}(x_0)$. A similar barrier function has been used in \cite[Lemma 2.3.10]{Ros23}.
	
	Let us discuss $c_{B.2.7}$ and $\delta$. In Corollary B.2.7 the author shows for any $\varepsilon \in(0, \alpha s)$ the existence of a function $\varphi \in C(\mathbb R^n) \cap C^\infty(\Omega)\cap H^s(\mathbb R^n)$ satisfying
	\begin{equation}
		\begin{cases}
			A_s\varphi \leq - d_x^{\epsilon - s} &\text{ in $\{0 < d_x < \delta\}$}, \\
			d_x^s \leq \varphi \leq C_{B.2.7} d_x^s &\text{ in $\Omega$}, \\
			\varphi = 0 &\text{ in $\Omega^c$},
		\end{cases}
	\end{equation}
	for some $C_{B.2.7}>1$ and $0 < \delta < 1$. Here $C_{B.2.7} = c_{B.2.7}^{-1}$. For the definition of $\varphi$ one first chooses a regularized distance function $\mathfrak{d}_x$, satisfying the bounds $d_x \leq \mathfrak d_x\leq C_Rd_x$ for some constant $C_R\geq 1$, see \cite[Definition 2.7.5]{Ros23}. Thereafter, the function $\phi$ is chosen as $\mathfrak{d}_x^s+\mathfrak{d}_x^{s+\eps'}$ for some $\eps'<\eps<\alpha s$. For our purpose, fix $\varepsilon = \alpha s_\star\slash 2$ and $\varepsilon' = \varepsilon\slash 2$. Going trough the proof one sees that we can choose
	\begin{equation*}
		C_{B.2.7} = C_R + C_R^2(\diam \Omega +1).
	\end{equation*} 
	For $\delta \in (0, 1)$ it is enough to choose it such that
	\begin{equation*}
		\delta^{\varepsilon - \varepsilon'} \leq \frac{c_{B.2.5}C_R^{\varepsilon' - 1}}{C_{B.2.1} + C_{B.2.5} +1}.
	\end{equation*}
	
	The constant $C_{B.2.1}$ depends on geometric quantities coming from the regularity of $\mathfrak d_x$, $\Lambda$, $d$ and $s_\star$, as one proves 
	\begin{equation*}
		|A_s \mathfrak d_x^s| \leq C_{B.2.1}\mathfrak d_x^{\alpha s - s} \text{ in } \Omega.
	\end{equation*}
	The dependencies of $C_{B.2.1}$ are straightforward to check by carefully following the proof in \cite[Proposition B.2.1]{Ros23}. It is only in the end of the proof of Proposition B.2.1 (i), where the author estimates $|A_s\mathfrak d_x^s |$ by decomposing into several integrals, that the author produces a factor of $1\slash (1-s)$. Note that this bears no problem as we normalize the operator $A_s$ with the factor $1-s$ as mentioned in the beginning of the remark.
	
	The constants $c_{B.2.5}$ and $C_{B.2.5}$ appear as one proves for $x_0\in \Omega$
	\begin{equation*}
		A_s \mathfrak d_x^{s+\varepsilon} \leq -c_{B.2.5}\mathfrak d_x^{\varepsilon - s} + C_{B.2.5} \quad\text{in $\Omega\cap B_1(x_0)$}.
	\end{equation*}
	Here $C_{B.2.5}$ only depends on the $C^{1, \alpha}$-regularity of $\mathfrak d_x$ and $c_{B.2.5}$ essentially depends on the constants from Lemma B.1.6 which in turn depend on Lemma B.1.5. Let us explain the constant $c_{B.1.6}$. In Lemma B.1.5 one relates $A_su$, for functions $u:\mathbb R^d\to\mathbb R$ such that $u(x) = \bar u(x\cdot e)$, $\bar u:\mathbb R\to\mathbb R$, $e\in \mathbb S^{d-1}$, by
	\begin{equation*}
		A_s u(x) = c_{B.1.5} (-\Delta)^s_{\mathbb R} \bar u(x\cdot e). 
	\end{equation*}
	In Lemma B.1.6 one chooses $\bar u(t) = t_+^{s+\varepsilon'}$ and, by a scaling argument, deduces
	\begin{equation*}
		(-\Delta)^s_{\mathbb R} \bar u(x\cdot e) = (-\Delta)^s_{\mathbb R} \bar u( 1) (x\cdot e)_+^{s+\varepsilon'} \quad\text{on $\{x\cdot e > 0\}$}.
	\end{equation*}
	Then $c_{B.1.6} := - c_{B.1.5}(-\Delta)^s_{\mathbb R} \bar u( 1)$. In Lemma B.1.6 it is shown that $ (-\Delta)^s_{\mathbb R} \bar u( 1)$  is strictly negative, and therefore $c_{B.1.6}> 0$. Both constants $c_{B.1.5}$ and $ (-\Delta)^s_{\mathbb R} \bar u( 1)$ will heavily depend on $s$. However, there are upper and lower bounds that are only depending on $s_\star, \varepsilon', \lambda$ and $\Lambda$. For $c_{B.1.5}$, these bounds are fairly easy to see. For $(-\Delta)^s_{\mathbb R} \bar u(1)$ this problem is more delicate. The idea is, the function $c(s, \varepsilon') =  -(-\Delta)^s_{\mathbb R} \bar u( 1)$ is continuous in $s\in(s_\star, 1)$ for fixed $\varepsilon' < s_\star$. In particular, the author showed that $c(s, \varepsilon') > 0$ for all choices $0 < \varepsilon' < s < 1$. So, it remains to show that $c(s, \varepsilon')$ remains bounded from above and below for $s$ close to 1, more precisely $s \in ( (1-\varepsilon') \vee 3\slash 4, 1)$. Then, one can use the infimum and supremum of $c(s, \varepsilon')$ w.r.t.\  $s\in (s_\star, 1)$ to get bounds independent of $s$. For $s$ close to 1, the idea is to separate \begin{equation*}
		-(-\Delta)^s_{\mathbb R} \bar u( 1)= (1-s)\int\limits_0^\infty \frac{(1+t)_+^{s+\eps'}+(1-t)_+^{s+\eps'}-2}{t^{1+2s}}\d t
	\end{equation*} into two integrals over $(0,1)$ and $(1,\infty)$ and treat them separately. We will call the integration variable in both terms $t$ from now on. Estimating the $(1,\infty)$ integral is easy, as it does not contain the singularity at $t=0$. For $t\in(0,1)$ one can use Taylor's formula to rewrite $(1\pm t)^{s+\varepsilon'}$. This will lead to the factor $s+\varepsilon' - 1$, which explains  the choice of $s > 1-\varepsilon'$. Moreover, one has to take care of the singularity appearing in $t = 1$, as the Taylor expansion yields a term $(1-t)^{s+\varepsilon-2}$. This can be treated by additionally separating the integral into $0<t<1\slash 2$ and $1\slash 2 < t < 1$ to split the singularities. Finally, there will appear a $(1-s)^{-1}$-term, which gets canceled by the $(1-s)$-term from the definition of $(-\Delta)^s_{\mathbb R}$.		
\end{remark}


\begin{remark}\label{rem:dyda_robust}
	Within the proof of \cite[Theorem 1]{Dyd06}, more precisely equation $(13)$, the author proves
	\begin{equation}\label{eq:dyda_constant}
		\int\limits_{\Omega}\int\limits_{\Omega} \frac{\abs{u(x)-u(y)}}{\abs{x-y}^{d+s}}\d y \d x \le C(d, \Omega, \eta, s)	\int\limits_{\Omega}\int\limits_{B_{\eta d_x}} \frac{\abs{u(x)-u(y)}^2}{\abs{x-y}^{d+s}}\d y \d x,
	\end{equation}
	for a Lipschitz domain $\Omega$ and any $\eta\in (0, 1)$. We will now discuss that the constant can be chosen depending on $s$ only through a lower bound $0<s_\star<s<1$. The main ingredients are \cite[Lemma 3, Proposition 5]{Dyd06}. In Lemma 3 it is proven that for any $p>0$, $\beta > 1$ there exists a constant $c>0$ such that for every nonnegative sequence $(a_n)_n$ one has
	\begin{equation*}
		\left(\sum_{k=1}^\infty a_k \right) \leq c\sum_{k=1}^\infty \beta^k a_k^p.
	\end{equation*}
	More precisely, the author shows that for any $m \geq 1$ we have
	\begin{equation*}
		\left(\sum_{k=1}^\infty a_k \right) \leq (2m)^{p-1}\sum_{k=1}^\infty (2^{\frac{p-1}{m}})^k a_k^p.
	\end{equation*}
	Thus, for $\beta >1$ any choice of the integer $m$ such that
	\begin{equation*}
		m \geq (p-1)\frac{\ln 2}{\ln\beta}
	\end{equation*}
	gives the desired estimate. In the proof of \eqref{eq:dyda_constant}, Lemma 3 is used for $\beta = (1+ \frac{\eta}{M})^s$, where $M>1$ is a number depending only on the Lipschitz constant of $\Omega$. This will lead to the fact that the constant $c$ depends on $s^{-1}$. In the last step of Proposition 5, the constant $C$ depends on $s^{-1}$ from the use of Lemma 3, some geometric quantities and the sum
	\begin{equation*}
		\sum_{k=1}^\infty \left(1-\frac{\eta^2}{M^2	}\right)^{2sk} = \frac{\left(1-\frac{\eta^2}{M^2	}\right)^{2s}}{1-\left(1-\frac{\eta^2}{M^2	}\right)^{2s}} \leq \frac{\left(1-\frac{\eta^2}{M^2	}\right)^{2s_\star}}{1-\left(1-\frac{\eta^2}{M^2	}\right)^{2s_\star}}.
	\end{equation*}
	The remaining steps in the proof of Proposition 5 and \eqref{eq:dyda_constant} are purely geometric and therefore do not lead to a factor depending on $s$.
\end{remark}
	
	\smallskip

\end{document}